\documentclass{article}
\usepackage{mathtools}
\usepackage{amsfonts}
\usepackage{amssymb}
\usepackage[utf8]{inputenc}
\usepackage[english]{babel}
\usepackage{amsthm}
\usepackage{cancel}
\usepackage{comment}

\usepackage{titlesec}
\titleformat{\paragraph}
{\normalfont\normalsize\bfseries}{\theparagraph}{1em}{}
\titlespacing*{\paragraph}
{0pt}{3.25ex plus 1ex minus .2ex}{1.5ex plus .2ex}

\usepackage[pdftex,pdfstartview=FitH,pdfborderstyle={/S/B/W 1}, %
colorlinks=true, linkcolor=blue, urlcolor=blue,citecolor=blue,%
pagebackref=true]{hyperref}

\theoremstyle{plain}
\newtheorem{theorem}{Theorem}
\numberwithin{theorem}{section}
\newtheorem{corollary}[theorem]{Corollary}
\newtheorem{proposition}[theorem]{Proposition}

\newtheorem{conv}[theorem]{Convention}

\theoremstyle{definition}
\newtheorem{definition}[theorem]{Definition}

\theoremstyle{remark}

\numberwithin{equation}{section}
\usepackage{graphicx}
\usepackage{pgfplots}
\usetikzlibrary{intersections}
\usetikzlibrary{calc}

\begin{document}

\title{RECURRENCE AND EMBEDDINGS IN PLANAR WIND-TREE MODELS}
\author{Chen Frenkel}
\date{}
\maketitle

We study periodic infinite billiards in the plane. We show that for rational models, some particular obstacles can be added periodically,
so that the billiard flow in the resulting table is recurrent in almost every direction. \\

\section{Introduction}

The subject of billiard dynamics has a long history and has been broadly studied. It is usually modeled on a `rational' polygonal table, where all the angles are rational multiples of $\pi$.
In that setting, the dynamical behavior is rather well understood and `nicely behaved': upon picking a starting point, the billiard flow is ergodic in almost every direction.

If we drop the assumption that the table is compact, then the picture is different. 
For example, we can construct an infinite table by taking the whole $\mathbb{R}^{2}$ plane, and placing (infinitely many) polygonal obstacles on it.
In this case, there isn't much we know about the billiard flow.

The classical wind-tree model was introduced long ago, in 1912 \cite{EhEh}.
As suggested, it is constructed by taking the plane $\mathbb{R}^{2}$, and placing identical rectangular obstacles periodically along the $\mathbb{Z}^{2}$-lattice. 
Surprisingly, only recently several results shed some light on the dynamics in this model.
These developments were achieved studying a different field, which initially seemed unrelated, dynamics on moduli spaces of flat surfaces.

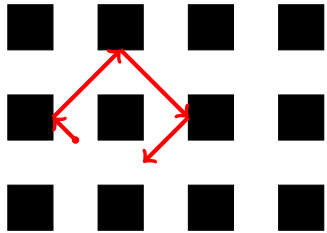
\begin{figure}[htbp]
\centering
\begin{tikzpicture}[every node/.style = {opacity = 1, black, above left}, scale=0.6]

\coordinate (rec1) at (0,0);
\draw [name path=rectangle1, fill=black]  (rec1) rectangle +(1,1);
\coordinate (rec2) at (2,0);
\draw [name path=rectangle2, fill=black]  (rec2) rectangle +(1,1);
\coordinate (rec3) at (4,0);
\draw [name path=rectangle3, fill=black]  (rec3) rectangle +(1,1);
\coordinate (rec4) at (6,0);
\draw [name path=rectangle4, fill=black]  (rec4) rectangle +(1,1);

\coordinate (rec5) at (0,-2);
\draw [name path=rectangle1, fill=black]  (rec5) rectangle +(1,1);
\coordinate (rec6) at (2,-2);
\draw [name path=rectangle2, fill=black]  (rec6) rectangle +(1,1);
\coordinate (rec7) at (4,-2);
\draw [name path=rectangle3, fill=black]  (rec7) rectangle +(1,1);
\coordinate (rec8) at (6,-2);
\draw [name path=rectangle4, fill=black]  (rec8) rectangle +(1,1);

\coordinate (rec9) at (0,-4);
\draw [name path=rectangle1, fill=black]  (rec9) rectangle +(1,1);
\coordinate (rec10) at (2,-4);
\draw [name path=rectangle2, fill=black]  (rec10) rectangle +(1,1);
\coordinate (rec11) at (4,-4);
\draw [name path=rectangle3, fill=black]  (rec11) rectangle +(1,1);
\coordinate (rec12) at (6,-4);
\draw [name path=rectangle4, fill=black]  (rec12) rectangle +(1,1);

\filldraw [red] (rec5) +(1.5,0) circle (2pt);
\draw [->,red,ultra thick] (rec5) +(1.5,0) -- +(1,0.5);
\draw [->,red,ultra thick] (rec5) +(1,0.5) -- +(2.5,2);
\draw [->,red,ultra thick] (rec2) +(0.5,0) -- +(2,-1.5);
\draw [->,red,ultra thick] (rec7) +(0,0.5) -- +(-1,-0.5);

\end{tikzpicture}
\caption{The original wind-tree model.}
\end{figure}

The wind-tree model is parameterized by the dimensions of the obstacles.
One chooses the parameters, a starting point $p$, and looks at the billiard flow $\mathcal{F}_{t}^{\theta}(p)$ in directions $\theta \in [0,2\pi)$.

Depending on the choice of specific parameters or directions, one can find all kind of dynamical properties.
For example \cite{HLT} it was shown that there are dense $G_{\delta}$ sets of parameters and directions (upon choosing $p$), for which $\mathcal{F}_{t}^{\theta}(p)$ is recurrent. That is, always returns to sets of positive measure.

In this work we want to study the general behavior of the billiard flow.
We look for results that hold for any choice of the parameters of the model, any starting point, and for Lebesgue almost every $\theta \in [0,2\pi)$.
Looking at the general behavior, rather surprising results were found: 

The billiard flow was shown to be non-ergodic \cite{FU}.
In \cite{DHL} the diffusion rate was found to be $\limsup _{t \rightarrow+\infty} \frac{\log d\left(p, \mathcal{F}_{t}^{\theta}(p)\right)}{\log t} = \frac{2}{3}$.
These results show that the billiard flow on the wind-tree model exhibits rather interesting, and somewhat unexpected, dynamical behaviors.

Nonetheless, there came a positive result in \cite{AH}, in which it is shown that the flow is recurrent.
The recurrence property will be the main focus of our paper.

\vspace{3mm}

\noindent
Some other wind-tree models have been considered. In \cite{FH}, recurrence was shown for models where the obstacles are periodic w.r.t.
a lattice of the form $(1, \lambda)  \mathbb{Z}+(0,1) \mathbb{Z}$. In \cite{Pa}, recurrence was established for models with right angled central-symmetric obstacles.
We will take the study to a much more general setting.

\vspace{5mm}

\noindent
In the present work, we tackle the recurrence problem for general periodic billiard models in the plane, which we will also call {\em planar wind-tree models}. 
A planar wind-tree model can be parameterized as $\mathrm{WT}(\Lambda, \Omega)$ using a lattice $\Lambda$ in $\mathbb{R}^2$ and a configuration $\Omega$. 
The configuration $\Omega \subset \mathbb{R}^2$ is of the form $\Omega = \overline{F \backslash \{O_i\}_{i=1}^{r_{\mathrm{o}}}}$ where $F$ is some chosen fundamental domain for $\Lambda$, and $O_i$ are disjoint polygonal obstacles (with connected boundary).
Then, the configuration $\Omega$ is placed periodically in the plane w.r.t.  $\Lambda$.

Note that the parameterization is not unique, and different choices of $\Omega$ are possible.
Like in the compact setting, we will confine the discussion to `rational' models as follows:
\label{rational:config:intro}
Write $e(\Omega) = (e_1, \dots, e_{r_e})$ for the edges of all the polygons in $\{O_i\}_{i=1}^{r_{\mathrm{o}}}$
and $\xi(\Omega) = (\xi_1, \dots, \xi_{r_e})$ for the angles $\xi_i \in[0, \pi)$ they make with any horizontal axis.
Then the model $\mathrm{WT}(\Lambda, \Omega)$ is called {\em rational} if all the differences $\xi_{i}-\xi_{j}$ are of the form $\frac{m}{n} \pi,$ with $n \in \mathbb{N}, m \in \mathbb{Z}$.

\vspace{3mm}

\noindent
We will add new obstacles to $\Omega$ to create a new configuration \\
$\widehat{\Omega}=\overline{F \backslash \{O_i\}_{i=1}^{r_{\mathrm{o}}+\widehat{r_{\mathrm{o}}}}}$, and discuss the resulting wind-tree model $\mathrm{WT}(\Lambda, \widehat{\Omega})$.\\
Then we show:

\begin{theorem}
\label{thm:main}
For any rational wind-tree model $\mathrm{WT}(\Lambda, \Omega)$ one can
add obstacles to $\Omega$ so that for the new configuration $\widehat{\Omega}$ the billiard flow in $\mathrm{WT}(\Lambda, \widehat{\Omega})$ is recurrent.
\end{theorem}

\vspace{2mm}

\noindent
The embedded obstacles in the proof of Theorem~\ref{thm:main} will be the same for all models, and will consist of a pair of $L$-shaped polygons.
Yet, their size, orientation and placement will depend on the specific model (though they can vary).

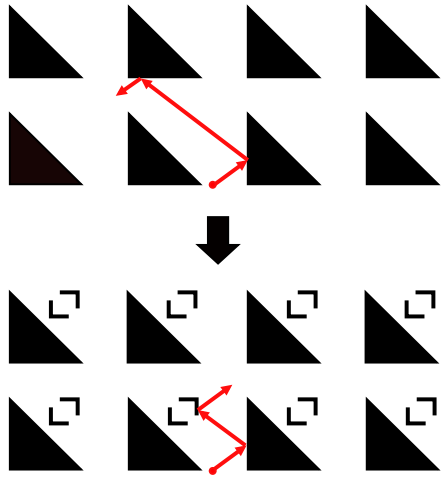
\begin{figure}[htbp]
\centering

\tikzset{every picture/.style={line width=0.75pt}} 

\begin{tikzpicture}[ scale=0.6, x=0.75pt,y=0.75pt,yscale=-1,xscale=1]

\draw  [fill={rgb, 255:red, 0; green, 0; blue, 0 }  ,fill opacity=1 ] (151,21) -- (211,80.33) -- (151,80.33) -- cycle ;
\draw  [color={rgb, 255:red, 0; green, 0; blue, 0 }  ,draw opacity=1 ][fill={rgb, 255:red, 0; green, 0; blue, 0 }  ,fill opacity=1 ] (251,21) -- (311,80.33) -- (251,80.33) -- cycle ;
\draw  [fill={rgb, 255:red, 0; green, 0; blue, 0 }  ,fill opacity=1 ] (351,21) -- (411,80.33) -- (351,80.33) -- cycle ;
\draw  [color={rgb, 255:red, 0; green, 0; blue, 0 }  ,draw opacity=1 ][fill={rgb, 255:red, 0; green, 0; blue, 0 }  ,fill opacity=1 ] (451,21) -- (511,80.33) -- (451,80.33) -- cycle ;
\draw  [fill={rgb, 255:red, 23; green, 5; blue, 5 }  ,fill opacity=1 ] (151,111) -- (211,170.33) -- (151,170.33) -- cycle ;
\draw  [fill={rgb, 255:red, 0; green, 0; blue, 0 }  ,fill opacity=1 ] (251,111) -- (311,170.33) -- (251,170.33) -- cycle ;
\draw  [fill={rgb, 255:red, 0; green, 0; blue, 0 }  ,fill opacity=1 ] (351,111) -- (411,170.33) -- (351,170.33) -- cycle ;
\draw  [fill={rgb, 255:red, 0; green, 0; blue, 0 }  ,fill opacity=1 ] (451,111) -- (511,170.33) -- (451,170.33) -- cycle ;
\draw  [fill={rgb, 255:red, 0; green, 0; blue, 0 }  ,fill opacity=1 ] (251,351) -- (311,410.33) -- (251,410.33) -- cycle ;
\draw  [line width=1.5]  (293,351.29) -- (308,351.29) -- (308,364.79) ;
\draw [color={rgb, 255:red, 255; green, 14; blue, 14 }  ,draw opacity=1 ][fill={rgb, 255:red, 172; green, 15; blue, 34 }  ,fill opacity=1 ][line width=1.5]    (321,411.29) -- (346.78,392.37) ;
\draw [shift={(350,390)}, rotate = 503.72] [fill={rgb, 255:red, 255; green, 14; blue, 14 }  ,fill opacity=1 ][line width=0.08]  [draw opacity=0] (9.29,-4.46) -- (0,0) -- (9.29,4.46) -- cycle    ;
\draw [color={rgb, 255:red, 255; green, 14; blue, 14 }  ,draw opacity=1 ][fill={rgb, 255:red, 172; green, 15; blue, 34 }  ,fill opacity=1 ][line width=1.5]    (349,390) -- (312.21,362.67) ;
\draw [shift={(309,360.29)}, rotate = 396.61] [fill={rgb, 255:red, 255; green, 14; blue, 14 }  ,fill opacity=1 ][line width=0.08]  [draw opacity=0] (9.29,-4.46) -- (0,0) -- (9.29,4.46) -- cycle    ;
\draw  [line width=1.5]  (300.5,371.54) -- (285.5,371.54) -- (285.5,358.04) ;
\draw [color={rgb, 255:red, 255; green, 14; blue, 14 }  ,draw opacity=1 ][fill={rgb, 255:red, 172; green, 15; blue, 34 }  ,fill opacity=1 ][line width=1.5]    (309,360.29) -- (334.78,341.37) ;
\draw [shift={(338,339)}, rotate = 503.72] [fill={rgb, 255:red, 255; green, 14; blue, 14 }  ,fill opacity=1 ][line width=0.08]  [draw opacity=0] (9.29,-4.46) -- (0,0) -- (9.29,4.46) -- cycle    ;
\draw  [fill={rgb, 255:red, 0; green, 0; blue, 0 }  ,fill opacity=1 ] (351,351) -- (411,410.33) -- (351,410.33) -- cycle ;
\draw  [line width=1.5]  (393,351.29) -- (408,351.29) -- (408,364.79) ;
\draw  [line width=1.5]  (400.5,371.54) -- (385.5,371.54) -- (385.5,358.04) ;
\draw  [fill={rgb, 255:red, 0; green, 0; blue, 0 }  ,fill opacity=1 ] (351,261) -- (411,320.33) -- (351,320.33) -- cycle ;
\draw  [line width=1.5]  (393,261.29) -- (408,261.29) -- (408,274.79) ;
\draw  [line width=1.5]  (400.5,281.54) -- (385.5,281.54) -- (385.5,268.04) ;
\draw  [fill={rgb, 255:red, 0; green, 0; blue, 0 }  ,fill opacity=1 ] (450,261) -- (510,320.33) -- (450,320.33) -- cycle ;
\draw  [line width=1.5]  (492,261.29) -- (507,261.29) -- (507,274.79) ;
\draw  [line width=1.5]  (499.5,281.54) -- (484.5,281.54) -- (484.5,268.04) ;
\draw  [fill={rgb, 255:red, 0; green, 0; blue, 0 }  ,fill opacity=1 ] (451,351) -- (511,410.33) -- (451,410.33) -- cycle ;
\draw  [line width=1.5]  (493,351.29) -- (508,351.29) -- (508,364.79) ;
\draw  [line width=1.5]  (500.5,371.54) -- (485.5,371.54) -- (485.5,358.04) ;
\draw  [color={rgb, 255:red, 252; green, 15; blue, 15 }  ,draw opacity=1 ][fill={rgb, 255:red, 255; green, 0; blue, 4 }  ,fill opacity=1 ] (319,411.5) .. controls (319,410.12) and (320.12,409) .. (321.5,409) .. controls (322.88,409) and (324,410.12) .. (324,411.5) .. controls (324,412.88) and (322.88,414) .. (321.5,414) .. controls (320.12,414) and (319,412.88) .. (319,411.5) -- cycle ;
\draw [color={rgb, 255:red, 255; green, 14; blue, 14 }  ,draw opacity=1 ][fill={rgb, 255:red, 172; green, 15; blue, 34 }  ,fill opacity=1 ][line width=1.5]    (322,171.29) -- (347.78,152.37) ;
\draw [shift={(351,150)}, rotate = 503.72] [fill={rgb, 255:red, 255; green, 14; blue, 14 }  ,fill opacity=1 ][line width=0.08]  [draw opacity=0] (9.29,-4.46) -- (0,0) -- (9.29,4.46) -- cycle    ;
\draw [color={rgb, 255:red, 255; green, 14; blue, 14 }  ,draw opacity=1 ][fill={rgb, 255:red, 172; green, 15; blue, 34 }  ,fill opacity=1 ][line width=1.5]    (351,150) -- (264.18,83.76) ;
\draw [shift={(261,81.33)}, rotate = 397.34000000000003] [fill={rgb, 255:red, 255; green, 14; blue, 14 }  ,fill opacity=1 ][line width=0.08]  [draw opacity=0] (9.29,-4.46) -- (0,0) -- (9.29,4.46) -- cycle    ;
\draw  [color={rgb, 255:red, 252; green, 15; blue, 15 }  ,draw opacity=1 ][fill={rgb, 255:red, 255; green, 0; blue, 4 }  ,fill opacity=1 ] (319,170.79) .. controls (319,169.41) and (320.12,168.29) .. (321.5,168.29) .. controls (322.88,168.29) and (324,169.41) .. (324,170.79) .. controls (324,172.17) and (322.88,173.29) .. (321.5,173.29) .. controls (320.12,173.29) and (319,172.17) .. (319,170.79) -- cycle ;
\draw [color={rgb, 255:red, 255; green, 14; blue, 14 }  ,draw opacity=1 ][fill={rgb, 255:red, 172; green, 15; blue, 34 }  ,fill opacity=1 ][line width=1.5]    (261,81.33) -- (243.28,93.71) ;
\draw [shift={(240,96)}, rotate = 325.07] [fill={rgb, 255:red, 255; green, 14; blue, 14 }  ,fill opacity=1 ][line width=0.08]  [draw opacity=0] (9.29,-4.46) -- (0,0) -- (9.29,4.46) -- cycle    ;
\draw  [fill={rgb, 255:red, 4; green, 0; blue, 0 }  ,fill opacity=1 ] (309,221.4) -- (317.5,221.4) -- (317.5,198) -- (334.5,198) -- (334.5,221.4) -- (343,221.4) -- (326,237) -- cycle ;
\draw  [fill={rgb, 255:red, 0; green, 0; blue, 0 }  ,fill opacity=1 ] (250,261) -- (310,320.33) -- (250,320.33) -- cycle ;
\draw  [line width=1.5]  (292,261.29) -- (307,261.29) -- (307,274.79) ;
\draw  [line width=1.5]  (299.5,281.54) -- (284.5,281.54) -- (284.5,268.04) ;
\draw  [fill={rgb, 255:red, 0; green, 0; blue, 0 }  ,fill opacity=1 ] (151,261) -- (211,320.33) -- (151,320.33) -- cycle ;
\draw  [line width=1.5]  (193,261.29) -- (208,261.29) -- (208,274.79) ;
\draw  [line width=1.5]  (200.5,281.54) -- (185.5,281.54) -- (185.5,268.04) ;
\draw  [fill={rgb, 255:red, 0; green, 0; blue, 0 }  ,fill opacity=1 ] (151,351) -- (211,410.33) -- (151,410.33) -- cycle ;
\draw  [line width=1.5]  (193,351.29) -- (208,351.29) -- (208,364.79) ;
\draw  [line width=1.5]  (200.5,371.54) -- (185.5,371.54) -- (185.5,358.04) ;

\end{tikzpicture}

\caption{$\mathrm{WT}(\Lambda, \Omega)$ (top) and $\mathrm{WT}(\Lambda,  \widehat{\Omega})$ (bottom).}
\end{figure}

\vspace{5mm}
\noindent
The billiard flow on a compact rational table can be studied by looking at the geodesic flow on a corresponding `translation surface' $X$.
Some non-compact models can also be studied using cover constructions of translation surfaces, termed `$\mathbb{Z}^{k}$-covers'.
$\mathbb{Z}^{k}$-covers are defined at the level of homology using a set of classes $\overline{\gamma}=\{[\gamma_{i}]\}_{i=1}^{k} \subset H_{1}(X, \mathbb{Z})$.
$\mathrm{WT}(\Lambda, \Omega)$ will be described as a $\mathbb{Z}^{2}$-cover of some surface $X$.

The dynamical behavior in $\mathbb{Z}$-covers is better understood,
and recurrence corresponds to a simple `no-drift' condition on $[\gamma_1]$ (\cite{HoW}).
The general case of recurrence in a $\mathbb{Z}^k$-cover with $k \ge 2$ is more complicated, as the no-drift condition is not sufficient.
This was demonstrated with a counterexample by Delecroix \cite{De}.

\vspace{5mm}
\noindent
In our work, we will show recurrence in $\mathrm{WT}(\Lambda,  \widehat{\Omega})$ by using a geometric criterion that was proved in \cite{AH}.
The criterion concerns the (stratified) moduli space of translation surfaces, and the Hodge bundle $H_{g}^{1}$ over it .
We will define these, and the Kontsevich-Zorich cocycle of $(X,\overline{\gamma})$ in $H_{g}^{1}$.
The criterion involves a splitting $H_{g}^{1} = V \oplus V^{\perp}$, with $\overline{\gamma} \subset V(X)$, that is invariant for the cocycle.
Then one needs to find a `good cylinder' --- a flat annulus in $X$ with core curve in $V^{\perp}(X)$.

For a given $\mathrm{WT}(\Lambda, \Omega)$ we will construct $\widehat{\Omega}$
so that $X(\Lambda, \widehat{\Omega})$ contains such a good cylinder.

\subsection{Organization of the paper}

We begin with Section 2, where we recall some background material concerning flat surfaces and their moduli spaces.
Section 3 discusses translation surfaces of infinite type, and $\mathbb{Z}^k$-translation covering constructions.
In Section 4 we formulate the unfolding procedure of planar wind-tree models.
Section 5 shows the recurrence criterion from \cite{AH}.
Then in Section 6 we give the proof of the main theorem.

\subsection{Acknowledgements}

This paper is part of the author's MSc Thesis at Tel Aviv University, under the supervision of Prof. Barak Weiss.
The author would like to express his deepest gratitude to Prof. Barak Weiss, for the his guidance and patience.
The partial support of grants BSF 2016256 and ISF 2095/15 is gratefully acknowledged.


\section{Background}
We begin with some background, concerning flat surfaces and their moduli spaces. We also present several constructions, to be used later.
For additional reading, see $\cite{FoMa},\cite{Yo},\cite{Zo}$.
We especially recommend the available chapters of the upcoming book $\cite{DVH}$ which discusses translation surfaces of infinite type.

\subsection{Translation surfaces and properties}
Translation structure is a stronger notion that that of complex structure.
There are three equivalent definitions for a translation surface --- a constructive definition, a geometric one and an analytic one.
The three definitions lead to some nice properties, which we will mention.

\paragraph{Constructive definition of translation surfaces}
\label{sec:trans:constr}
We call {\em generalized polygon} a compact and connected subset $O \subset \mathbb{R}^{2}$ whose boundary is a finite union of piecewise smooth segments (edges).
Note that the boundary need not be connected.
We will assume this generalization when referring to polygons, unless we are in the context of billiard unfolding. There, when it is important for the edges to be linear, we will use the term `Euclidean polygon'.

We write each edge of $O$ as $e = \left( \gamma, \upsilon \right)$ where 
$\gamma: [0,1] \rightarrow \mathbb{R}^{2}$ is the corresponding path,
and $\upsilon \in \{\pm1\}$ indicates whether or not $e$ is positively oriented w.r.t. $O$. 

A translation surface will be constructed from a finite configuration of (generalized) polygons $\Omega=\left\{O_{i}\right\}_{i=1}^{r_{\mathrm{o}}}$,
by identifying their edges in some manner.

We denote by $e(\Omega)=\left(e_{1}, \ldots, e_{2d}\right)$ the set of edges of all the polygons, which we assume has an even number of elements.
We also denote by $s(\Omega) = \left\{ s_1, \dots, s_{r_s} \right \}$
the unique set of vertices from $\left(\gamma_{1}(0),\gamma_{1}(1),\ldots, \gamma_{2d}(0), \gamma_{2d}(1)\right)$.

\begin{definition}
\label{def:trans:constr}
(Constructive). A {\em translation surface} $X$, or a {\em flat surface}, consists of a configuration $\Omega$ of polygons, an edge set $e(\Omega)$,
and a (pairing) map $\mathfrak{T}:e(\Omega) \rightarrow e(\Omega)$, $\mathfrak{T}^2 = I$,
so that for each $\mathfrak{T}(\gamma_{i},\upsilon_i) =  (\gamma_j,\upsilon_j)$ we have $\upsilon_j = -\upsilon_i$ and 
$\gamma_j(t) = \gamma_i(t) + \tau_{i}$ where $\tau_{i} \in \mathbb{R}^{2}$.
Taking the quotient $\Omega/\mathfrak{T}$ by the relation $\mathfrak{T}$ on the edges, we get the translation surface $X$.
\end{definition}

\vspace{1mm}
\noindent
We will always assume the construction results in a connected surface.
In later sections, $\gamma$ will be used also for marking general paths in $X$, so in places of ambiguity we will also write $\gamma_{e_i}$. 
We consider the subset $e^+(\Omega) \subset e(\Omega)$ of positively oriented edges, so that $e(\Omega) = e^+(\Omega) \cup \mathfrak{T}(e^+(\Omega))$. 
We write $e(\Omega) = (e_1, \dots, e_d, e_{d+1}, \dots, e_{2d})$, with $e_1, \dots, e_d \in e^+(\Omega)$, $ e_{d+i} = \mathfrak{T}(e_i)$.


\paragraph{Cone singularities}
\label{cone:sing}

In the constructive definition, the images of the vertices $s(\Omega)$ under $\mathfrak{T}$ are cone singularities.
That is, each singularity has a neighbourhood $U$ which is a branched coverings of the unit disk ${ \mathbb{P}:U \rightarrow \Delta}$,
and a local coordinate map $(U,w)$ such that $\mathbb{P} w^{-1} = z^{k+1}$ for some $k \in \mathbb{N}$.

\begin{conv}\label{convention}
\label{conv:sing:conf}
{\rm 
When a surface $X$ is given together with a configuration $\Omega$ of polygons, as in Definition~\ref{def:trans:constr},
we will take the singularity set $\Sigma$ as $\Sigma =  s(\Omega) / \mathfrak{T}$.
}
\end{conv}

That is, our choice will then depend on $\Omega$ and we might include vertices with $k_{\sigma} = 0$, which we call regular singularities or marked points.
We denote by $\Sigma_\mathrm{s} \subset \Sigma$ the non-regular singularities.

All the properties concerning points of $X\backslash \Sigma$ also hold for regular singularities, 
but it will be convenient to keep them in $\Sigma$ when we construct a basis for homology based on $\Omega$ in Subsection~\ref{basis:homology}.


\paragraph{Geometric and analytic definitions}

The geometric definition of a translation surface is given by a translation atlas on $X\backslash \Sigma$,
$\left\{U_{\alpha},w_{\alpha}\right\}_{\alpha}$, whose transition functions are all translation maps in $\mathbb{C}$:
${ w_{\beta} w_{\alpha}^{-1}(z)=z+\tau_{\alpha, \beta} }$.

The analytic definition of translation surfaces is given by a pair $(X,\omega)$ of a Riemann surface $X$ and an abelian (holomorphic) differential $\omega$.

Considering the Constructive Definition~\ref{def:trans:constr}, locally in $X \backslash \Sigma$ the abelian differential can be taken as $\omega = dz$,
where $z$ is the natural coordinate in $\mathbb{C}$. Around a singularity $\sigma \in \Sigma$ of order $k$, $\omega = d(z^{k+1}) =  (k+1) z^{k} d z$.
So a cone singularity of order $k$ corresponds to a zero of order $k$ of $\omega$.

\vspace{5mm}
\noindent
Using the Riemann–Hurwitz formula, we get the following relation for the orders of zeros of $\omega$:
\begin{equation}
\label{formula:sum:sing}
\sum_{\sigma \in \Sigma_{s}} k_{\sigma} = 2g - 2,
\end{equation}
where $g$ is the genus of $X$.


\paragraph{Properties of translation surfaces}

Having a translation atlas on $X\backslash \Sigma_{s}$ allows pulling the vector field $\frac{\partial}{\partial y}$ and the euclidean metric from the plane, both being translation invariant under the transition maps $dw_{\beta} w_{\alpha}^{-1}(z) = dz$.
We get a global `north' direction field and a flat metric on $X\backslash \Sigma_{s}$.

We use here $\Sigma_\mathrm{s}$ instead of $\Sigma$, as we will consider these also on regular singularities.
Points of $\Sigma_{s}$ are singularities of the metric, and $X$ is the metric completion of $X \backslash \Sigma_{s}$.

The vector field $v_{\theta}(z)$ of direction $\theta$ is the vector field such that locally $\omega(v_{\theta}(z)) =  e^{i \theta}$.
The geodesic flow in the north direction $\phi_{t}(x)$ is called the {\em Translation flow}, the {\em Linear flow} or the {\em Straight-line flow}.
$\phi_{t}(x)$ is defined for all $t$, as long as it does not go through a singularity in $\Sigma_\mathrm{s}$.
We denote by $\phi_{t}^{\theta}(x)$ the geodesic flow in direction $\theta$.

\vspace{5mm}
\noindent
There are two geometric objects on a translation surface, concerning the translation flow.
A {\em Saddle connection} is a geodesic segment, which goes through $X \backslash \Sigma_\mathrm{s}$, except for its endpoints which are in $\Sigma_\mathrm{s}$. 
A {\em Cylinder} is a maximal flat annulus foliated by homotopic closed geodesics in $X \backslash \Sigma_\mathrm{s}$.


\paragraph{Translation equivalence maps and translation covering maps}

A morphism of translation surfaces is called {\em translation equivalence map}, 
and it is a biholomorphism ${ f:(X,\omega_X) \rightarrow (Y,\omega_Y) }$ such that $f^*\omega_Y = \omega_X$.
Such a map acts in local flat coordinates by translations ${f(z)=z+\tau}$.

A {\em translation covering map} between two translation surfaces $\widetilde{X}$ and $X$
is a cover map $\mathbb{P}:(\widetilde{X},\widetilde{\omega})\rightarrow (X,\omega)$  such that $\widetilde{\omega} = \mathbb{P}^*\omega$.

\vspace{3mm}
\noindent
For a subgroup of $\pi_1(X\backslash\Sigma, x_0)$ with some $x_0 \in X\backslash\Sigma$, one constructs the cover $\widetilde{X}$ by pulling back the translation structure on $\widetilde{X} \backslash \widetilde{\Sigma}$
and looking at the metric completion $\widetilde{X}$, and the map $\mathbb{P}:(\widetilde{X},\widetilde{x}_0)\rightarrow (X,x_0)$.

This construction also sets the ground for non-compact translation surfaces, which we will discuss later in further detail.
We ignore some issues here. For instance, $\widetilde{\Sigma}$ are the cone singularities of $\widetilde{X}$ and might be infinitely ramified (logarithmic singularity).
Since this will not happen in the wind-tree model cover constructions, we ignore it (for a complete discussion, see for example $\cite{DVH}$).

We will be interested in the correspondence between the linear flow $\phi_{t}^{\theta}(x)$ on $X$ and its lift $\widetilde{\phi}_{t}^{\theta}(\widetilde{x})$ to $\widetilde{X}$.


\paragraph{Affine diffeomorphism and $\mathrm{GL}(2, \mathbb{R})$-action}

An {\em affine diffeomorphism} $f: X \rightarrow Y$ between translation surfaces is a diffeomorphism that in local (flat) coordinates of $X$ and $Y$ acts
by affine maps 
$\left(\begin{array}{c}{x}  \\ {y}\end{array}\right)  \mapsto  \left(\begin{array}{cc}{a} & {b} \\ {c} & {d}\end{array}\right)  \left(\begin{array}{c}{x}  \\ {y}\end{array}\right) + \left(\begin{array}{c}{\tau_{\alpha}}  \\ {\tau_{\beta}}\end{array}\right) $. 

We can describe Y by looking at $f^*\omega_Y = A \omega_X$. 
If $\{U_{\alpha}, w_{\alpha}\}_{\alpha}$ is the initial translation structure of $X$, then the pulled-back structure from $Y$ is $\left\{U_{\alpha}, Aw_{\alpha}\right\}_{\alpha}$ which we denote by $AX$.

\vspace{5mm}

\noindent
Thus, this also allows us to define the action of $\mathrm{GL}(2, \mathbb{R})$ on a translation surface, $X \rightarrow AX$.
In terms of the notions involved in the constructive definition with a configuration $\Omega=\left\{O_{i}\right\}_{i=1}^{r_{\mathrm{o}}}$, one acts on the polygons $\left\{AO_{i}\right\}_{i=1}^{r_{\mathrm{o}}}$.
The $\mathrm{SL}(2, \mathbb{R})$-action, which is area and orientation preserving, will be of special interest.

\vspace{5mm}

\noindent
We let $\text{Aff}(X)$ denote the group of affine automorphisms on $X$.
We will also consider the group $\text{Aff}(X,\Sigma)$, whose elements fix $\Sigma$ (not necessarily pointwise).
We denote by $\text{Aff}^+(X,\Sigma)$ the subgroup of orientation preserving maps.

\paragraph{Intersection form and orientation}
\label{intersection:orientation}

The wind-tree cover constructions will be defined in terms of intersections.
We denote by $H_{1}\left(X, \Sigma, \mathbb{Z}\right)$ the first homology relative to the $\Sigma$,
and by $ i: H_{1}(X, \Sigma, \mathbb{Z}) \times H_{1}(X \backslash \Sigma, \mathbb{Z}) \rightarrow \mathbb{Z} $ the Poincaré-Lefschetz intersection pairing.

\vspace{5mm}
\noindent
A map $f \in \text{Aff}(X,\Sigma)$ affects the intersection form in the following way:
\begin{equation}
\label{form:intersection:orientation}
i\left(f_{*}\left[\gamma_{1}\right], f_{*}\left[\gamma_{2}^{*}\right]\right) = \text{sign}(\mathrm{det}(A)) \, i([\gamma_1],[\gamma_2^{*}]).
\end{equation}
for $\left[\gamma_{1}\right] \in H_{1}(X, \Sigma, \mathbb{Z}),\left[\gamma_{2}^{*}\right] \in H_{1}(X \backslash \Sigma, \mathbb{Z})$.

\paragraph{Holonomy map}
\label{sec:holonomy}

An abelian differential $\omega$ defines a map $\mathrm{hol}_{\omega} \in \operatorname{Hom}\left(H_{1}\left(X, \Sigma, \mathbb{Z}\right), \mathbb{C}\right)$ by
 $\mathrm{hol}_{\omega}(\left[\gamma\right]) := \left[\gamma\right] \mapsto \int_{\gamma} \omega$.
It is called {\em the Holonomy map}, and we also denote it for short by $\omega([\gamma]) := \mathrm{hol}_{\omega}([\gamma])$.

\vspace{4mm}
\noindent
We make two observations, that will be helpful later on.

\vspace{3mm}
\noindent
First, we see that $\omega$ is defined by the values of $\mathrm{hol}_{\omega}$ on a basis $\{[\gamma_i]\}_{i=1}^{2g+|\Sigma|-1} \subset H_{1}\left(X, \Sigma, \mathbb{Z}\right) \cong \mathbb{Z}^{2g+|\Sigma|-1}$; 
the values $\{\omega([\gamma_i])\}_{i=1}^{2g+|\Sigma|-1}$ are called {\em periods}.

\vspace{3mm}
\noindent
Secondly, if a translation surface is given by a configuration $\Omega$, an edge ${ e_i=\left(\gamma_i, \upsilon_i \right) \in e(\Omega) }$
defines an homology class $h(e_i) = [\gamma_i] \in H_{1}\left(X, \Sigma, \mathbb{Z}\right)$.
Note that ${ h(e_i) = h(\mathfrak{T}(e_i)) }$, and $\omega(h(e_i)) = \gamma(1)-\gamma(0)$.

We will see in Section~\ref{basis:homology} that the homology classes $h(e_i)$ with $e_i \in e(\Omega)$
form a basis for $ H_{1}\left(X, \Sigma, \mathbb{Z}\right)$ (up to some technicalities which will be discussed).

\paragraph{Cutting and gluing in parallel}
\label{sec:cut:glue}

For a translation surface $X$ and a configuration $\Omega=\left\{O_{j}\right\}_{j=1}^{r_{\mathrm{o}}}$,
there are two operations that can be done to $\Omega$, so that the resulting surface is still the same $X$.
The two correspond to the usual cutting and gluing operations on surfaces, but adapted to the context of translation surfaces.

As we will have only a minor use of them, and being quite intuitive, we just sketch the details. 

\vspace{3mm}
\noindent
{\em cutting in parallel} is done by cutting a polygon $O \in \Omega$ along a simple piecewise smooth path $\gamma \subset O$,
splitting it into two components, and producing two new isometric edges that can be identified together.
One needs to verify that either $\gamma(1) = \mathfrak{T}(\gamma(0))$, or that the endpoints of $\gamma$ are vertices.

{\em gluing in parallel} is done by gluing $e \in e(O_1)$ and $\mathfrak{T}(e) \in e(O_2)$, into a new polygon $O$.
One needs to verify that the gluing can be done in the plane (with no overlappings).

\subsubsection{Half-translation surfaces}

\paragraph{Definition of half-translation surfaces}

We consider another object, called {\em half-translation surface}, which enjoys properties similar to those of translation surfaces.
Like earlier, there are three equivalent definitions.
For the constructive definition, consider a configuration like those discussed in Subsection~\ref{sec:trans:constr}.

\begin{definition}
\label{def:half:trans:constr}
(Constructive). A {\em half-translation surface} $X$ consists of a configuration $\Omega$, orientation $\Upsilon$ of the plane, edge set $e(\Omega)$, 
and a pairing ${ \mathfrak{T}:e(\Omega) \rightarrow e(\Omega) }$, $\mathfrak{T}^2 = I$, such that for each $\mathfrak{T}(e_i) = e_j$ exactly one of the following holds:
\begin{enumerate}
  \item $\upsilon_j = -\upsilon_i$ and $\gamma_j(t) = \gamma_i(t) + \tau_i$.
  \item $\gamma_j(t) = \gamma_i(1-t) + \tau_i$.
\end{enumerate}
Taking the quotient $\Omega/\mathfrak{T}$ by the relation $\mathfrak{T}$ on the edges, we get the half-translation surface $X$.
\end{definition}

\vspace{3mm}
\noindent
This time, the vertices $\Sigma_s$ might give {\em half-integer} cone singularities,
having a neighbourhood isometric to $\left(\Delta_r^{*}, g_{\alpha}\right)$ where 
$g_{\alpha}=(d r)^{2}+(\frac{k+1}{2} r d \theta)^{2}$, $ k \in \mathbb{N}$.

The geometric definition is given by a {\em half-translation atlas} $\left\{U_{\alpha},w_{\alpha}\right\}_{\alpha}$ on $X\backslash\Sigma_s$,
whose transition functions are all half-translation maps in $\mathbb{C}$: ${ w_{\beta} w_{\alpha}^{-1}(z)=\pm z+\tau_{\alpha, \beta} }$.

The analytic definition is given using a quadratic differential,
a tensor of the local form $q = p_{\alpha}(w_\alpha)(dw_\alpha)^2$, with $p_\alpha$ holomorphic,
so that on $U_{\alpha} \cap U_{\beta}$ we have $p_{\beta}(w_{\beta})(\frac{dw_{\beta}}{dw_{\alpha}})^2 = p_{\alpha}(w_{\alpha})$.
A half-integer cone singularity of order $k$ correspond to a zero of order $k$ of $q$. A zero of order $-1$ is actually a simple pole of $q$.

\vspace{3mm}
\noindent
Some of the properties of a translation surface can be generalized to half-translation surfaces.
We will be interested mainly in the $\mathrm{SL}(2, \mathbb{R})$-action on a half-translation surface, which is similarly described using Definition~\ref{def:half:trans:constr}.

\paragraph{Double cover, involution}
\label{half:double:inv}

A half-translation surface $(X,q)$ might not longer have global direction fields, as those need not be invariant through half-translations.
In such case, we say that $q$ is non-orientable.

Given $(X,q)$ with a non-orientable $q$, we construct a double cover $\widehat{X}$ of $X$, which is a translation surface.
$(\widehat{X},\omega)$ is made of two copies of $\Omega_X$, $\Omega_{\widehat{X}} = (\Omega_X,\iota \Omega_X)$,
where $\iota \in \text{Aff}(\widehat{X})$ is a rotation by $\pi$, $d\iota = -I$.
It has the following identifications:
If $\mathfrak{T}_X(e_i)=e_j$ is identified by a half-translation $\left( \gamma_j(t)= \gamma_i(1-t)+\tau_i  \right)$,
then $\mathfrak{T}_{\widehat{X}}(e_i)=\iota e_i$ and $\mathfrak{T}_{\widehat{X}}(e_j)=\iota e_j$.

\vspace{5mm}
\noindent
For a translation surface $(X,\omega)$, a map $\iota \in \text{Aff}(X)$ as above $d\iota = -I$ will be called an involution.
If $X$ has an involution $\iota$ then $X/\iota$ is a half translation surface with a non-orientable $q$, and $X$ is its double cover.

\paragraph{Symplectic splitting of homology}
\label{half:split:hom:sec}

Let $(X,\omega)$ be a translation surface with involution $\iota$.
There is a symplectic splitting of homology 
\begin{equation}
H_1(X,\mathbb{C}) = H_1^{\iota+}(X,\mathbb{C}) \oplus H_1^{\iota-}(X,\mathbb{C}),
\end{equation}
where $H_1^{\iota+}(X,\mathbb{C})$ is the subspace of classes invariant under $\iota_*$, and $H_1^{\iota-}(X,\mathbb{C})$ is the subspace of anti-invariant classes.

One can also consider a similar splitting for relative homology
$H_{1}(X, \Sigma,\mathbb{C})$, given that $\iota \in \text{Aff}(X,\Sigma)$.
One then has direct sum decompositions
\begin{equation}
\label{half:split:hom}
\begin{aligned}
H_1(X, \Sigma,\mathbb{C}) = { H_1^{\iota+}(X, \Sigma,\mathbb{C}) \oplus H_1^{\iota-}(X, \Sigma,\mathbb{C}) }, \\
H_1(X \backslash \Sigma,\mathbb{C}) = H_1^{\iota+}(X \backslash \Sigma,\mathbb{C}) \oplus H_1^{\iota-}(X \backslash \Sigma,\mathbb{C}).
\end{aligned}
\end{equation}
Now for $[\gamma_1^{+}] \in  H_1^{\iota+}(X, \Sigma,\mathbb{C})$ and $[{\gamma_2^{*}}^{-}] \in H_1^{\iota-}(X \backslash \Sigma,\mathbb{C})$ we have
${ i([\gamma_1^{+}],[{\gamma_2^{*}}^{-}]) = 0 }$ (and also for  $[\gamma_1^{-}] \in  H_1^{\iota-}(X, \Sigma,\mathbb{C})$ and $[{\gamma_2^{*}}^{+}] \in H_1^{\iota+}(X \backslash \Sigma,\mathbb{C})$).
In what follows, a splitting as in formula~(\ref{half:split:hom}) will also be called symplectic.

\vspace{5mm}
\noindent
We can identify 
$H_{1}^{\iota+}(X, \mathbb{C}) \cong H_{1}(X/\iota, \mathbb{C})$, using the exact sequence\\
${ 0 \rightarrow   H_{1}^{\iota+}(X, \mathbb{C})   \rightarrow  H_{1}(X, \mathbb{C})  \stackrel{\mathbb{P}_*}{\longrightarrow}   H_{1}(X/\iota, \mathbb{C}) \rightarrow 0}$.
When looking at relative homology, there is a delicate consideration where simple poles lift to regular points (which we do not mark).
So, if we put $\widetilde{\Sigma}(X)  = \mathbb{P}^{-1}(\Sigma(X/\iota)) \backslash \{\widetilde{\sigma} \,|\, k_\sigma =-1 \}$,
then this time it is the anti-invariant subspace
\begin{equation}
H_{1}^{\iota-}(X, \widetilde{\Sigma}, \mathbb{C}) \cong H_{1}(X/\iota, \Sigma, \mathbb{C}).
\end{equation}


\subsection{Rational billiards}
\subsubsection{Rational billiards unfolding}
\label{sec:billiard:unfold}
A billiard in the plane is a dynamical model, in which a particle moves in a straight line, and upon hitting an obstacle it reflects elastically. 
The billiard flow in certain polygonal tables $P$  can be modeled as the linear flow on a corresponding translation surface.
The idea behind the construction is that instead of reflecting the trajectory upon collision, one can reflect the table $P$ in the plane, and let the particle move in the same direction.
However, one needs to choose polygons for which this process is finite.

The billiard unfolding construction is finite for {\em rational billiards}, and is called {\em the Katok-Zemliakov billiard unfolding}. 
We define both, using some of the notations of Subsection~\ref{sec:trans:constr}.

\vspace{5mm}
\noindent
Let $P$ be a Euclidean polygon, and $e(P)=\left(e_1, \ldots, e_{r_e}\right)$ be its edge set.
In contrast with generalized polygons, here the edges must be linear for the billiard model to be meaningful.
To each edge $e_{i}$ we associate the reflection $\rho_{e_{i}}$ in the plane along the line going through $\gamma_i$.

Denote by $\xi(P)=\left(\xi_{1}, \ldots, \xi_{r_e}\right)$ the set of angles the edges make with some chosen horizontal axis,
like in the Introduction~\ref{rational:config:intro}. Denote by ${ \xi_{\text{diff}}(P)=\left\{ \xi_{i}-\xi_{j} \right\}_{i,j} }$
the set of differences.

A polygon is called {\em rational} if its set of differences is of the form $\xi_{\text{diff}}(P) = \left\{  \frac{m_i}{n_i} \pi \,|\, n_i \in \mathbb{N}, m_i \in \mathbb{Z} \right\}_{i}$.
For a rational polygon, the number $n(P) := \mathrm{l.c.m.}(\left\{   n_i   \right\}_i)$ is called {\em the unfolding constant}.

\vspace{5mm}
\noindent
Let $D_n$ denote the dihedral group of $2n$ elements, $D_n = \left\{  \rho_1, \dots, \rho_{2n}      \right\}$, where $\rho_1 = I$ is the identity, and $n$ is the unfolding constant.
We will also write $\left\{\rho_{1}, \ldots, \rho_{2 n}\right\} = \left\{  \theta_{1}, \ldots, \theta_{n}, \rho\theta_{1}, \ldots, \rho\theta_{n}  \right\}$
where $\theta_i$ are rotation elements, and $\rho$ is a reflection element.

The associated translation surface $X(P)$ is constructed from the configuration $\Omega=\left\{ \rho_i P  \right\}_{i=1}^{2n}$,
where each time we reflect $P$ in a different copy of the plane. 
For each edge $e_j \in e(P)$, $\rho_i e_j$ is defined by $\rho_i e_j = \left(   \rho_i \gamma_j, \upsilon_j \cdot \upsilon_{\rho_i}  \right)$
where\\ $\upsilon_{\rho_i}=\left\{\begin{array}{ll}{+1,} & {\text{if }\rho_i = \;\;\theta_i} \\ {-1,} & {\text{if }\rho_i = \rho\theta_{i-n}.}\end{array}\right.$

Now, the edge set of $\Omega$ is $e(\Omega) = D_n e(P).$
The identifications are given by $\mathfrak{T}(\rho_i e_j) = \rho_{i}\rho_{e_j}\rho_{i}^{-1} \rho_i e_j = \rho_i \rho_{e_j}e_j$.
With these identifications, one gets the associated translation surface $X(P)$.

\vspace{3mm}
\noindent
The billiard flow on $P$ lifts to a linear flow on $X$, and a linear flow on $X$ projects to the billiard flow on $P$.
Note though that $X(P)$ is not a cover of $P$.
We will study billiard flows by examining the associated linear flows.

\subsubsection{Induced affine automorphisms}
\label{sec:billiard:action}

Given a polygon $P$, the unfolding group $D_{n}$ and the unfolded surface $X(P)$, the configuration $\Omega=\left\{\rho_{i} P\right\}_{i=1}^{2 n}$
is invariant under the maps $\rho_i$ that permute the copies of $P$. They are affine automorphisms of $X$. Also, putting $\Sigma = D_n s(P) / \mathfrak{T}$, we have $\rho_i \in \text{Aff}(X,\Sigma)$.

\vspace{5mm}
\noindent
Formula~(\ref{form:intersection:orientation}) now takes the form:
\begin{corollary}
\label{cor:billiard:sign}
Let $[\gamma_{1}] \in H_{1}(X, \Sigma, \mathbb{Z}),[\gamma_{2}^{*}] \in H_{1}(X \backslash \Sigma, \mathbb{Z})$.\\
For $\{ \rho_{i}= \theta_{i} \}_{i=1}^n$,   $ \quad\quad\,\,\,\,  i(\rho_{i*}[\gamma_{1}],\rho_{i*}[\gamma_{2}^{*}]) =  \quad \!\!  i([\gamma_{1}],[\gamma_{2}^{*}])$.\\
For $\{ \rho_{i}=\rho \theta_{i-n}\}_{i=n+1}^{2n}$,  $           i(\rho_{i*}[\gamma_{1}],\rho_{i*}[\gamma_{2}^{*}]) =     	   -i([\gamma_{1}],[\gamma_{2}^{*}])$.
\end{corollary}

\subsubsection{Parking garages}
\label{sec:parking}

Following \cite{CW}, we can generalize the unfolding procedure of Section~\ref{sec:billiard:unfold} from polygons to objects called `parking garages'.

Let $P$ be a compact connected surface with boundary, together with an immersion $w_P:P \rightarrow \mathbb{R}^2$.
Assume that $w_P(\partial P)$ is a finite union of linear segments, $w_P(\partial P) = (e_1, \dots, e_{r_e})$. Then $P$ is called {\em a parking garage}.
It can also be described using the Constructive Definition~\ref{def:trans:constr}, but where a subset of linear edges is not paired, and serve as boundary.
The unfolding procedure can then be applied to $P$.

The edges are $e(P) = w_{P}(\partial P)$. We can define $\xi(P),n(P)$ similarly.
The copies $\rho_i P$ are defined using post composition on the immersion $\rho_i \circ w_P$.
The unfolding now follows the earlier procedure, using the identification law ${ \mathfrak{T}(\rho_i e_j) = \rho_i \rho_{e_j}e_j }$.

\vspace{5mm}
\noindent
An interesting use of a parking garage $P$ is to consider $P$ as a sort of translation surface with piece-wise smooth boundary.

A {\em translation surface with boundary} $X_b$ is a translation surface $X$ with a number of open disks removed from its interior $X \backslash \{\Delta_i\}_i$.
Most of the properties we have seen earlier can be generalized to translation surfaces with boundary.
In a parking garage $P$ the boundary is only piece-wise smooth. Still, cover constructions (as we will encounter later), which are essentially topological, hold just as well.

\vspace{5mm}
\noindent
Let us also introduce a more restrictive type of parking garages, which we will encounter in the context of the wind-tree model (this will be important in Section \ref{cover:boundary}).

Assume a translation surface is given like in the Constructive Definition~\ref{def:trans:constr},
just so that one removes open disks $\Delta_i$ from the interior of the polygons $O_j$ of $\Omega$.
That is, the boundary components are not allowed to intersect any edges or vertices.
This way, one carries out the edge identifications while ignoring the boundary.
We call this a {\em simple translation surface with boundary}.

Similar to the above, but where the boundary components are polygons, gives us a {\em simple parking garage}.

\vspace{5mm}
\noindent
In the wind-tree unfolding, we will work with a simple parking garage.
This will allow us to use cover constructions of a translation surface, while also being able to apply an unfolding procedure.

\subsubsection{Configuration and basis for homology}
\label{basis:homology}

We now show how, given a translation surface with a configuration $\Omega$ (and via some manipulations), the homology classes represented by the edges give a basis for homology $H_1(X,\Sigma,\mathbb{Z})$.

This construction will be helpful in Section~\ref{sec:z:d:cover}, for formula~(\ref{z:k:cover:constr}), where we describe $\mathbb{Z}^k$-covers in terms of intersections in $H_1(X,\Sigma_\mathrm{s},\mathbb{Z})$.
The construction will let us easily describe the wind-tree unfolding process later.
This is where Convention~\ref{conv:sing:conf} is needed.

\vspace{5mm}
\noindent
We begin with a simple case. We take $\Omega = {O}$ to be a single polygon $O$ with connected boundary.
We will find a basis for $H_{1}\left(X, \Sigma_{\mathrm{s}}, \mathbb{Z}\right)$.

As $h(e_i) = h(\mathfrak{T}(e_i))$, we focus on $e^+(O)$.
For each $e_i \in e^{+}(O)$ define the cycle $\gamma^{*}_{e_i}$ as a simple loop starting at the midpoint of $\mathfrak{T}(e_i)$,
going in the interior of $O$ and ending at the midpoint of $e_i$, without crossing any other edges. Note that this is well defined up to homotopy,
and $\omega([\gamma^{*}_{e_i}]) = -\tau_i, [\gamma^{*}_{e_i}] \in H_{1}\left(X \backslash \Sigma_{\mathrm{s}}, \mathbb{Z}\right)$.

Letting $e_1, \dots, e_d \in e^+(O)$ be the positively oriented edges, we look at $h(e_1), \dots, h(e_d)$.
We have $i(h(e_i),[\gamma^{*}_{e_i}]) = 1$, and $i(h(e_i),[\gamma^{*}_{e_j}]) = 0$ for $i \ne j$. 
So $\{h(e_i)\}_{i=1}^{d}$ is a basis of $H_1(X,\Sigma_\mathrm{s},\mathbb{Z})$, and $\{[\gamma^{*}_{e_i}]\}_{i=1}^{d}$ is the dual basis of $H_{1}\left(X \backslash \Sigma_{\mathrm{s}}, \mathbb{Z}\right)$.

\vspace{5mm}
\noindent
Passing to general configurations $\Omega=\left\{O_{j}\right\}_{j=1}^{r_{\mathrm{o}}}$, 
and allowing additional marked points $\Sigma \cup \Sigma_{\mathrm{m}}$, introduces some technicalities.
We will sketch the details, but focus on the main conclusion for later.

First, one needs to glue in parallel the polygons into a connected surface with boundary $P$ along $2(r_{\mathrm{o}}-1)$ edges. 
(If these are Euclidean polygons, then $P$ is a parking garage).
For each marked point $\sigma_{\mathrm{m}} \in \Sigma_{\mathrm{m}}$, we need to add an edge between it and some vertex $\sigma \in \Sigma$.
If the boundary of $P$ is not connected, we need to add edges between two vertices on different boundary components.

One can let the additional edges pass in the interior of $P$, without intersecting the boundary, and keeping $P$ connected.
It is then possible to find dual edges, like earlier, showing that the edges give a  basis of $H_1(X,\Sigma \cup \Sigma_{\mathrm{m}},\mathbb{Z})$,
and also a symplectic basis of $H_1(X\backslash(\Sigma \cup \Sigma_{\mathrm{m}}),\mathbb{Z})$.

The main fact for formula~(\ref{z:k:cover:constr}) later is that upon choosing a subset of edges $e_{\mathrm{C}} \subset e(\Omega)$,
if one can glue $\Omega$ to a connected surface along $e(\Omega) \backslash e_{\mathrm{C}}$, 
then the classes $h(e_{\mathrm{c}}), e_{\mathrm{c}} \in e_{\mathrm{C}}$ are part of the symplectic basis.


\subsection{Tiling of the plane}
\label{tilling}

In this section, we show how a tiling of the plane, defined by a lattice $\Lambda$ in $\mathbb{R}^2$ and a fundamental domains $F \subset \mathbb{R}^2$,
can be described using a translation surface.

\vspace{5mm}
\noindent
When refering to fundamental domains $F \subset \mathbb{R}^2$ we restrict ourselves to domains that are simply connected and with simple piecewise smooth boundary.
Elements of $\Lambda$ translate $F$ to exhaust the plane.
The boundary of $F$ consists of pairs of isometric segments $\gamma_i, \mathfrak{T}(\gamma_i)$ that differ by a translation 
$\tau_{i} \in \Lambda$, and with different orientation w.r.t $F$. 

Following Subsection~\ref{sec:cut:glue}, we see that $F$ is a configuration $\Omega = \{F\}$ of a translation surface $X(O)$.
Note that since every vertex of $X$ is summed up to $2\pi$ in the plane, $X$ is a torus. 

\vspace{3mm}
\noindent
Now we can describe the tiling using $X$.
We take copies $F(l_1,l_2)$,  $l_1,l_2 \in \mathbb{Z}$ of $F$, and glue them as follows: 

Pick two different edges $e_1,e_2  \in e^+(F)$, so that $\tau_1,\tau_2$ are linearly independent. 
As $\Lambda$ is a 2-dimensional lattice, write for each $\tau_i$, $\tau_i = a_1^i\tau_1 + a_2^i\tau_2$, $ a_1^i, a_2^i \in \mathbb{Z}$.
Now glue
\begin{equation}
\label{form:lattice:identifications}
e_i(l_1,l_2) \sim \mathfrak{T}(e_i)(l_1 + a_1^i, l_2 + a_2^i).
\end{equation}
Specifically, we glue $e_1(l_1,l_2) \sim \mathfrak{T}(e_1)(l_1 + 1,l_2)$, and $e_2(l_1,l_2) \sim \mathfrak{T}(e_1)(l_1,l_2 + 1)$.
This generates the plane.


\subsection{Moduli spaces of flat surfaces}
\label{moduli:space:section}

\subsubsection{Moduli space of flat surfaces}
Fix a compact Riemann surface $R$ of genus $g \ge 1$. 
Choose an (ordered) discrete subset $\Sigma \subset R, \Sigma = (\sigma_1,\dots,\sigma_{|\Sigma|})$, which will represent singularities and marked points.
Denote by $\operatorname{Diff}(R,\Sigma)$ the group of diffeomorphisms of $R$ that fix each $\sigma_i$.
Let $\operatorname{Diff}_{0}(R,\Sigma) \subset \operatorname{Diff}(R,\Sigma)$ be the subgroup of elements that are (smoothly) isotopic to the identity (through maps fixing each $\sigma_i$).

Recalling formula~(\ref{formula:sum:sing}), we choose an ordered set $\kappa=\left(k_{1}, \dots, k_{|\Sigma|}\right)$, so that $\sum_{i=1}^{|\Sigma|} k_{i}=2 g-2$.
Here we also allow $k_i = 0$, that is, marked points.
We call $\kappa$ {\em the combinatorial datum}.
We also denote $n = 2 g+|\Sigma|-1$.

\vspace{5mm}
\noindent
The stratified Teichmüller space of abelian differentials $\mathcal{T} \mathcal{H}(\kappa)$ consists of triplets $(X,\omega,f)$
where two points $(X_1,\omega_1,f_1),(X_2,\omega_2,f_2)$ are said to be equivalent if
$f_2 \circ f_1^{-1} : X_1 \rightarrow X_2$ is isotopic (through diffeomorphisms sending each $f_1(\sigma_i)$ to $f_2(\sigma_i)$) to a diffeomorphism $h: X_1 \rightarrow X_2$
such that $h^*\omega_2 = \omega_1$.

\vspace{5mm}
\noindent
$\mathcal{T} \mathcal{H}(\kappa)$ can be endowed with a topology and complex affine structure.
On a small neighbourhood of a point $\left(X, \omega_{X}, f_{X}\right) \in U \subset \mathcal{T} \mathcal{H}(\kappa)$,
one can identify elements of the homology groups
$H_{1}\left(Y, \Sigma(\omega_Y), \mathbb{C}\right)$ and $H_{1}\left(X, \Sigma(\omega_X), \mathbb{C}\right)$
for $\left(Y, \omega_{Y}, f_{Y}\right) \in U$, using the Gauss-Manin connection.

Choosing a basis $\left(  [\gamma_1], \dots, [\gamma_n]  \right)$ of $H_{1}(R, \Sigma, \mathbb{Z})$,
the map
$\Theta_U: U \rightarrow \mathbb{C}^n$, $\Theta_U((Y,\omega_Y,f_Y)) = \left( \omega_{Y}(f_{Y *} [\gamma_{1}]), \dots, \omega_{Y}(f_{Y *} [\gamma_{n}]) \right)$,
taking each surface to its periods, gives local coordinates.

\vspace{5mm}
\noindent
The stratified moduli space of abelian differentials $\mathcal{H}(\kappa)$ consists of points $(X,\omega)$,
where $X$ comes with a labeling of the singularities $\Sigma(X)$, 
so that two points $(X_1,\omega_1),(X_2,\omega_2)$ are equivalent if
there exists a diffeomorphism ${ h: X_1 \rightarrow X_2 }$, with $h(\sigma_i(X_1)) = \sigma_i(X_2)$, for which $h^*\omega_2 = \omega_1$. 
We will simply refer to it as a {\em stratum}.

\vspace{5mm}
\noindent
The group $\operatorname{Diff}(R, \Sigma)$ acts on $\mathcal{TH}(\kappa)$ by $\left(X, \omega, f \right) \rightarrow\left(X, \omega, f \circ g^{-1}\right)$ 
for $g \in \operatorname{Diff}(R, \Sigma)$. 
The discrete group $\operatorname{Mod}(R, \Sigma) = \operatorname{Diff}(R, \Sigma) / \operatorname{Diff}_{0}(R, \Sigma)$ is the mapping class group.
We then have $\mathcal{H}(\kappa) = \mathcal{TH}(\kappa) / \operatorname{Mod}(R, \Sigma)$. That is, we forget the marking.

The stratum $\mathcal{H}(\kappa)$ has a complex affine orbifold structure,
as the $\operatorname{Mod}(R, \Sigma)$-action on $\mathcal{TH}(\kappa)$ can be shown to be properly discontinuous. 

\subsubsection{Affine measure and normalization}
\label{mod:space:norm}

One applies several normalizations on $\mathcal{H}(\kappa)$, to allow endowing it with a finite measure. 

Under `normalization of orientation', we use the subgroup $\operatorname{Mod}^{+}(R, \Sigma)$ of orientation-preserving diffeomorphisms.
We denote by $\mathcal{TH}^{+}(\kappa)$ the space of points $(X,\omega,f) \in \mathcal{TH}(\kappa)$ with $f$ orientation preserving, and where isotopy is via $\operatorname{Diff}_{0}^{+}(R,\Sigma)$. 
We set $\mathcal{H}^{+}(\kappa) = \mathcal{TH}^{+}(\kappa) / \operatorname{Mod}^{+}(R, \Sigma)$.

Next, we apply `normalization of area', where ${ \mathcal{H}_{1}(\kappa) \subset \mathcal{H}^{+}(\kappa) }$
consists of abelian differentials corresponding to translation surfaces of area 1.

Now, one can use the period coordinates to pull the standard volume form from $\mathbb{C}^n$ to $\mathcal{TH}^{+}(\kappa)$.
This is possible since the transition maps are affine and orientation-preserving.
We endow $\mathcal{TH}^{+}(\kappa)$ with the associated Lebesgue measure. 
We also note that the measure is invariant under the $\operatorname{Mod}^{+}(R, \Sigma)$-action,
so we obtain a Lebesgue-measure on $\mathcal{H}^{+}(\kappa)$ as well, which we denote by $\lambda$.
We mark by $\lambda_{1}$ the derived measure on $\mathcal{H}_{1}(\kappa)$.

\vspace{3mm}
\noindent
A fundamental result by Masur and Veech is the following \cite{Ma1,Ve}:
\begin{theorem}
The total mass of $\lambda_{1}$ is finite.
\end{theorem}
\noindent
Having a finite measure sets the ground for discussing ergodic actions.

\subsubsection{$\mathrm{SL}(2, \mathbb{R})$-action  and geodesic flow}

Using the $\mathrm{SL}(2,\mathbb{R})$-action on a translation surface, we can define an ${ \mathrm{SL}(2,\mathbb{R}) \text{-action} }$ on $\mathcal{TH}_{1}(\kappa),\mathcal{H}_{1}(\kappa)$.
It is applied by post-composition on local coordinates,\\ 
${ A (X,\omega,f_X) = (A X, A\omega, f_A \circ f_X) }$, $A \in \mathrm{SL}(2,\mathbb{R})$.

Since the $\mathrm{SL}(2,\mathbb{R})$-action is orientation and area preserving, the action is defined on $\mathcal{TH}_{1}(\kappa),\mathcal{H}_{1}(\kappa)$.
The $\mathrm{SL}(2,\mathbb{R})$-action is continuous and preserves the measures $\lambda,\lambda_{1}$.

\vspace{5mm}
\noindent
Thus, one can study the dynamics of the $\mathrm{SL}(2,\mathbb{R})$-action on strata. 
Of special interest is the action of the subgroup $g_{t}=\left(\begin{array}{cc}{e^{t}} & {0} \\ {0} & {e^{-t}}\end{array}\right)$, $t \in \mathbb{R}$,
called {\em the Teichmüller geodesic flow}.
An important result by Masur \cite{Ma1} is:
\begin{theorem}
The $g_t$-action on $\mathcal{H}_{1}(\kappa)$ and thus also the $\mathrm{SL}(2,\mathbb{R})$-action is ergodic w.r.t. $\lambda_{1}$.
\end{theorem}

\noindent
Studying the dynamics on the stratum can help study the dynamics on an individual translation surface, like in the next result:
\begin{theorem}
\label{thm:masur:unique:ergod}
(Masur criterion \cite{Ma2})
If the $g_t$-orbit of a surface ${ (X,\omega) \in \mathcal{H}_{1}(\kappa) }$ is non-divergent, 
that is, $g_{t_n}(X,\omega) \rightarrow (Y,\omega_Y)$ for some sequence ${ t_n \rightarrow \infty, t_n \in \mathbb{R} }$,
then the linear flow in the vertical direction on $X$ is uniquely ergodic.
\end{theorem}

\noindent
Actually, this result is even more general, and holds in the non-stratified moduli space of flat surfaces.
Using this result, and a number of further more ideas, it is possible to show the following:

\begin{theorem}
[\cite{KMS}]
For every translation surface $X$, the linear flow is ergodic in almost every direction.
\end{theorem}

\subsubsection{The Hodge bundle and the Kontsevich-Zorich cocycle}

We look at the extended real Hodge bundle $\widehat{H}_{g}^{1}$ over $\mathcal{TH}_{1}(\kappa)$, with fibers $H_{1}(X,  \Sigma(\omega), \mathbb{R})$ over $(X,\omega,f)$.
We endow the bundle with the Gauss-Manin flat connection, whose flat trivializations we have seen earlier by identifying nearby fibers. 
This bundle is trivial, and diffeomoprhic to ${ \widehat{H}_{g}^{1} \cong \mathcal{TH}_{1}(\kappa) \times H_{1}(R, \Sigma, \mathbb{R}) }$.
We denote the trivial cocycle $ \widehat{B}^{(t)}: \widehat{H}_{g}^{1} \rightarrow \widehat{H}_{g}^{1} $,
$ \widehat{B}^{(t)}((X,\omega,f), [c])=\left(g_{t}(X,\omega,f), [c]\right) $ by $\widehat{B}^{(t)}$.

\vspace{3mm}
\noindent
The mapping class group $\mathrm{Mod}^{+}(R, \Sigma)$ acts on $\widehat{H}_{g}^{1}$. We have described its action by pre-composition on $\mathcal{TH}_{1}(\kappa)$, which gives the stratum $\mathcal{H}_{1}(\kappa)$,
while it acts by symplectic automorphisms on $H_{1}(R, \Sigma, \mathbb{R})$. So we look at the real Hodge bundle $H_{g}^{1} = (\mathcal{TH}_{1}(\kappa) \times H_{1}(R, \mathbb{R})) / \mathrm{Mod}^{+}(R, \Sigma)$.
Now, the map defined by applying $B^{(t)} = \widehat{B}^{(t)} / \mathrm{Mod}^{+}(R, \Sigma)$ on $H_{g}^{1}$ is called the {\em Kontsevich–Zorich cocycle} (KZ-cocycle for short).

\vspace{5mm}
\noindent
The KZ-cocycle will be of main interest in our paper. 
One looks at the cocycle over loci in $\mathcal{H}_{1}(\kappa)$ given by $\mathrm{SL}(2,\mathbb{R})$-orbit closures of points,
and on invariant sub-bundles. 

\subsubsection{Affine measures and orbit closures}

One of the main problems in the dynamics on the strata, is to classify $\mathrm{SL}(2,\mathbb{R})$-orbit closures and $\mathrm{SL}(2,\mathbb{R})$-ergodic probability measures. 
Luckily, there is an important result concerning this.

\vspace{3mm}
\noindent
We will call {\em affine probabilty measure} any $\mathrm{SL}(2,\mathbb{R})$-ergodic probability measure $\mu$ on $\mathcal{H}_{1}(\kappa)$ 
which is given by pulling back the Lebesgue measure (like we did in Section~\ref{mod:space:norm}).

The breakthrough papers of Eskin-Mirzakhani \cite{EM}, and Eskin-Mirzakhani-Mohammadi \cite{EMM}, give us the following Ratner-like classification result:
\begin{theorem}
\label{EMM}
The orbit closure of any point $(X,\omega) \in \mathcal{H}_{1}(\kappa)$, $\mathcal{M} = \overline{\mathrm{SL}(2, \mathbb{R}) X}$,
is the support of a unique affine probabilty measure $\mu_{\mathcal{M}}$.
\end{theorem}

\noindent
This result also gives further details on the geometry of the support of affine measures, which we omitted.
We will use just one application of this far-reaching result.
Let us call a direction $\theta \in [0,2 \pi)$ {\em Birkhoff generic} for $X$ if
\begin{equation}
\lim _{T \rightarrow \infty} \frac{1}{T} \int_{0}^{T} f\left(g_{t} e^{i \theta} \omega\right) d t=\int_{\mathcal{M}} f d \mu_{\mathcal{M}}.
\end{equation}
for every $f \in C_c(\mathcal{H}_{1}(\kappa))$.
Then we have the following
\begin{theorem}
[\cite{CE}]
\label{Chaika-Eskin}
Let $(X,\omega) \in \mathcal{H}_{1}(\kappa)$, with $\mathrm{SL}(2,\mathbb{R})$-orbit closure $\mathcal{M}$. 
Then, Lebsgue almost every direction $\theta \in [0,2 \pi)$ is Birkhoff generic for $X$.
\end{theorem}

\subsubsection{Strata of quadratic differentials}

We can also discuss the strata of non-orientable quadratic differentials. Using the double covers of half-translation surfaces, it is an easy generalization of the strata of abelian differentials.
This will allow us to produce invariant sub-bundles for the KZ-cocycle.

We use notations similar to those in Section~\ref{mod:space:norm}. The combinatorial datum now encodes the orders of zeros and simple poles of quadratic differentials.
Put $n = 2g + |\Sigma| - 2$.
Recall the definitions of double covers and involutions from Section~\ref{half:double:inv}.

\vspace{3mm}
\noindent
The stratified Teichmüller space of quadratic differentials $\mathcal{T} \mathcal{Q}(\kappa)$ consists of triplets $(X,q,f)$,
where $f:R\rightarrow X$ is a diffeomorphism,
and $q$ is a non-orientable quadratic differential on $X$ whose zeros are $f(\Sigma)$, and each $f(\sigma_i)$ is a zero of order $k_i$.
Two points $(X_1,q_1,f_1),(X_2,q_2,f_2)$ are said to be equivalent if
$f_2 \circ f_1^{-1} : X_1 \rightarrow X_2$ is isotopic (through maps sending each $f_1(\sigma_i)$ to $f_2(\sigma_i)$) to a diffeomorphism $h: X_1 \rightarrow X_2$
so that $h^*q_2 = q_1$.
The stratified moduli space of quadratic differentials is then $\mathcal{Q}(\kappa) = \mathcal{TQ}(\kappa) / \operatorname{Mod}(R, \Sigma)$.

\vspace{4mm}
\noindent
We fix a non-orientable quadratic differential $q_0$ on $R$ matching $\kappa$, and let $\widehat{R}$ be its double-cover with combinatorial datum $\widetilde{\kappa}$.
We look at the Teichmüller space $\mathcal{TH}(\widetilde{\kappa})$.
One can find a locus $\mathcal{L}_{\mathcal{TQ}(\kappa)} \subset \mathcal{TH}(\widetilde{\kappa})$
which corresponds to $\mathcal{TQ}(\kappa)$, and identify $\mathcal{Q}(\kappa)$ with its image in the moduli space $\mathcal{L}_{\mathcal{Q}(\kappa)}$.

Each $(X,q,f) \in \mathcal{TQ}(\kappa)$ has a double cover $(\widehat{X},\omega,\widetilde{f}) \in \mathcal{L}_{\mathcal{TQ}(\kappa)}$ 
with ${ \widetilde{f}: \widehat{R} \rightarrow \widehat{X} }$ a lift of $f$.
We identify the subspace ${ H_{1}^{\iota-}(\widehat{X}, \widetilde{\Sigma}(\omega), \mathbb{C}) \subset H_{1}(\widehat{X}, \widetilde{\Sigma}(\omega), \mathbb{C}) }$ 
with $H_{1}(X, \Sigma(q), \mathbb{C})$, as in Section~\ref{half:split:hom:sec}.

\vspace{4mm}
\noindent
Using $\mathcal{L}_{\mathcal{TQ}(\kappa)}, \mathcal{L}_{\mathcal{Q}(\kappa)}$ we get complex structure on $\mathcal{TQ}^{1}(\kappa), \mathcal{Q}_{1}(\kappa)$.
The normalization of orientation and area to $\mathcal{TQ}^{1}(\kappa), \mathcal{Q}_{1}(\kappa)$ are done similarly,
and they are endowed with Lebesgue measures, introduced as before. 

The $\mathrm{SL}(2, \mathbb{R})$-action on $\mathcal{TQ}^{1}(\kappa), \mathcal{Q}_{1}(\kappa)$ is defined by post composition,
$A (X,q,f_X) = (A X, Aq, f_A \circ f_X)$, $A \in \mathrm{SL}(2,\mathbb{R})$. We identify it with the action on $\mathcal{L}_{\mathcal{TQ}_{1}(\kappa)},\mathcal{L}_{\mathcal{Q}_{1}(\kappa)}$. It is continuous, and preserves the Lebesgue measures.

\vspace{5mm}
\noindent
We look at the real extended Hodge bundle $\widehat{Q}_{g}^{1}$ over $\mathcal{TQ}^{1}(\kappa)$, with fibers $H_{1}(X, \Sigma(q), \mathbb{R})$ over $(X,q,f)$.
It corresponds to the sub-bundle $\widehat{H}_{g}^{1,\iota-}(\mathcal{L}_{\mathcal{TQ}(\kappa)})$ over $\mathcal{L}_{\mathcal{TQ}(\kappa)}$
with fibers $H_{1}^{\iota-}(\widehat{X}, \widetilde{\Sigma}(\omega), \mathbb{R})$. 

As before, we denote $Q_{g}^{1} = \left( \mathcal{TQ}^{1}(\kappa) \times H_{1}(R,\Sigma, \mathbb{R}) \right) / \operatorname{Mod}^{+}(R, \Sigma)$,
and $\widehat{B}^{(t)}: \widehat{Q}_{g}^{1} \rightarrow \widehat{Q}_{g}^{1}$,
${ \widehat{B^{(t)}}\left((X,q,f), c \right)=\left(g_{t}(X,q,f), c\right) }$. 
Then the KZ-cocycle on $Q_{g}^{1}$ is defined by
$B^{(t)} = \widehat{B}^{(t)} / \operatorname{Mod}^{+}(R, \Sigma)$.

Identifying the KZ-cocycle of $(X,q)$ in $Q_{g}^{1}$ with that of $(\widehat{X},\omega)$ in $ H_{g}^{1,\iota+}(\mathcal{L}_{\mathcal{Q}_1(\kappa)})$,
We get the following:

\begin{corollary}
\label{cor:quad:hodge:split}
There is a symplectic splitting of the Hodge bundle $H_{g}^{1}(\mathcal{L}_{\mathcal{Q}_1(\kappa)})$,
$ H_{g}^{1}(\mathcal{L}_{\mathcal{Q}_1(\kappa)}) = H_{g}^{1,\iota+}(\mathcal{L}_{\mathcal{Q}_1(\kappa)})  \oplus H_{g}^{1,\iota-}(\mathcal{L}_{\mathcal{Q}_1(\kappa)})$,
invariant for the KZ-cocycle.
Specifically, having ${ (X,\omega) \in \mathcal{H}_{1}(\widetilde{\kappa}) }$ with involution $\iota$ and orbit closure $\mathcal{M}$,
we can take the KZ-cocycle on the sub-bundle $H_{g}^{1,\iota+}(\mathcal{M})$.
\end{corollary}

\section{$\mathbb{Z}^k$-covers}

\subsection{Infinite translation surfaces}

We generalize the definition of a translation surface to {\em translation surfaces of infinite type}, surfaces which no longer need be compact. 

\begin{definition}
\label{def:infinite:trans:constr}
(Constructive). Let $\Omega_{\infty} = \{O_i\}_{i=1}^{\infty}$ be a countable configuration of polygons.
A translation surface of infinite type $X_{\infty}$ consists of a configuration $\Omega_{\infty}$ and a pairing $\mathfrak{T} : e(\Omega_{\infty}) \rightarrow e(\Omega_{\infty})$
like in Definition~\ref{def:trans:constr}. We look at $\Omega_{\infty} / \mathfrak{T}$. We assume that the resulting surface is connected.
Then  $\Omega_{\infty} / \mathfrak{T}$ is a translation surface of infinite type $X_{\infty}$.
\end{definition}

In this definition, we allowed ourselves to ignore some issues which do not arise in the context of the wind-tree models, like infinitely ramified singularities.
For brevity, we will also use the term `infinite translation surface'. The analytic and geometric definitions can be generalized as well, by dropping the compactness requirement.

\vspace{5mm}
\noindent
We can also introduce infinite parking garages.\\
An infinite parking garage is an immersion $w_{P_{\infty}} : P_{\infty} \rightarrow \mathbb{R}^{2}$, where $P_{\infty}$ is a connected and closed surface with boundary,
so that for each compact subset $K \subset \mathbb{R}^2$, $w_{P}(\partial P)  \cap K$ is a finite union of linear segments. If $w_{P_{\infty}}$ is an embedding, we also call it `infinite polygon'.
We will usually describe an infinite parking garage using the above constructive definition, where a subset of $e(\Omega_{\infty})$ consisting of linear edges, is not paired and serves as boundary.

For example, the wind-tree model $\mathrm{WT}(\Lambda, \Omega)$ in $\mathbb{R}^2$ considered in the Introduction, is an infinite polygon,
with the boundary consisting of the obstacles.

The billiard unfolding procedure can also be carried similarly, if $P_{\infty}$ is given constructively,
though we would need to require $\xi(P_{\infty})$ to be finite. 
Only then we can check if $P_{\infty}$ is rational, and proceed as before,
using the identification law $\mathfrak{T}(\rho_i e_j) = \rho_i \rho_{e_j}e_j$.

\subsection{$\mathbb{Z}^k$-covers}
\label{sec:z:d:cover}

For a cover of a translation surface, we would like to have a nice description, in terms of Definition~\ref{def:infinite:trans:constr}, using the configuration of the finite surface.
That is the case for $\mathbb{Z}^k$-covers.

\begin{definition}
\cite{HoW}
Let $X$ be a translation surface.
Pick a $k$-tuple of classes $\overline{\gamma} = ([\gamma_1], \dots, [\gamma_k]), [\gamma_i] \in H_1(X,\Sigma,\mathbb{Z})$.
Assume $\{[\gamma_i]\}_i$ are linearly independent. Consider the homomorphism $\phi_{\overline{\gamma}}:\pi_1(X\backslash \Sigma, x_0) \rightarrow \mathbb{Z}^k$,
$ \varphi \rightarrow ( i([\gamma_1],[\varphi]),  \, \dots, \,   i([\gamma_k],[\varphi])) $.
Take the cover $(\widetilde{X}_{\overline{\gamma}},\widetilde{x}_0)$ corresponding to $\text{ker}(\phi_{\overline{\gamma}}) \subset \pi_{1}(X \backslash \Sigma, x_0)$.
Then $\widetilde{X}_{\overline{\gamma}}$ is called a {\em $\mathbb{Z}^k$-cover of $X$}.
\end{definition}
The terminology comes from the fact that the deck group of the cover is isomorphic to $\mathbb{Z}^k$.


\begin{figure}[htbp]
\begin{tikzpicture}[ every node/.style = {opacity = 1, black, above left}]

\coordinate (rec1) at (0,0);
\draw [name path=rectangle1]  (rec1) rectangle +(2,2);
\draw [->,red,decorate,decoration={snake,amplitude=2.9mm,segment length=5mm,post length=1mm}](rec1) +(1,0) -- +(1,2);

\draw [->] (2.5,1) -- (3.5,1);

\coordinate (rec2) at (4,0);
\draw [name path=rectangle2_edge1]  (rec2) +(0,0) -- +(1,0);
\draw [name path=rectangle2_edge2]  (rec2) +(0,0) -- +(0,2);
\draw [name path=rectangle2_edge3]  (rec2) +(0,2) -- +(1,2);
\draw [decorate,decoration={snake,amplitude=2.9mm,segment length=5mm,post length=1mm}](rec2) +(1,0) -- +(1,2);
\coordinate (rec2_s) at (4+0.5,0);
\draw [name path=rectangle2_edge1]  (rec2_s) +(1,2) -- +(2,2);
\draw [name path=rectangle2_edge2]  (rec2_s) +(2,2) -- +(2,0);
\draw [name path=rectangle2_edge3]  (rec2_s) +(2,0) -- +(1,0);
\draw [decorate,decoration={snake,amplitude=2.9mm,segment length=5mm,post length=1mm}](rec2_s) +(1,0) -- +(1,2);

\draw [->] (7,1) -- (8,1);

\coordinate (rec3) at (8+0.5,0);
\draw [decorate,decoration={snake,amplitude=2.9mm,segment length=5mm,post length=1mm}](rec3) +(0,0) -- +(0,2);
\draw [name path=rectangle3_edge1]  (rec3) +(0,2) -- +(2,2);
\draw [name path=rectangle3_edge2]  (rec3) +(1,0) -- +(1,2);
\draw [name path=rectangle3_edge3]  (rec3) +(0,0) -- +(2,0);
\draw [decorate,decoration={snake,amplitude=2.9mm,segment length=5mm,post length=1mm}](rec3) +(2,0) -- +(2,2);

\coordinate (rec4) at (0,-3);
\draw [name path=rectangle4_edge1]  (rec4) +(0,0) -- +(10,0);
\draw [name path=rectangle4_edge1]  (rec4) +(0,2) -- +(10,2);
\node[] at  (0,-1.18) {\ldots};
\node[] at  (0,-3.16) {\ldots};
\node[] at  (10.8,-1.18) {\ldots};
\node[] at  (10.8,-3.16) {\ldots};
\draw [decorate,decoration={snake,amplitude=2.9mm,segment length=5mm,post length=1mm}](rec4) +(1,0) -- +(1,2);
\draw [decorate,decoration={snake,amplitude=2.9mm,segment length=5mm,post length=1mm}](rec4) +(3,0) -- +(3,2);
\draw [decorate,decoration={snake,amplitude=2.9mm,segment length=5mm,post length=1mm}](rec4) +(5,0) -- +(5,2);
\draw [decorate,decoration={snake,amplitude=2.9mm,segment length=5mm,post length=1mm}](rec4) +(7,0) -- +(7,2);
\draw [decorate,decoration={snake,amplitude=2.9mm,segment length=5mm,post length=1mm}](rec4) +(9,0) -- +(9,2);

\end{tikzpicture}
\caption{A $\mathbb{Z}$-cover defined by a cycle on the torus.}
\end{figure}
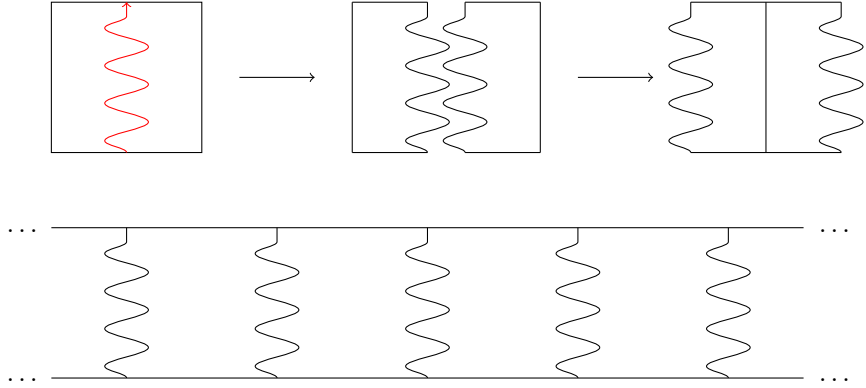

\vspace{5mm}
\noindent
For a $\mathbb{Z}^k$-cover we can describe its construction as in Definition~\ref{def:infinite:trans:constr}.\\
This is where we need the construction of homology basis using a configuration, from Section~\ref{basis:homology}.
As in that section, we modify $\Omega$ by gluing in parallel the polygons into one connected surface with boundary, and possibly add edges.
We still denote the new configuration by $\Omega$ (and soon we will see why).
We work with the configuration $\Omega_{\mathbb{Z}^k} = \{\Omega(\overline{l}) = \Omega(l_1, \dots, l_k)  | \, \overline{l} \in \mathbb{Z}^k \}$.
Recall that $h(e_i) = h(\mathfrak{T}(e_i))$, ${ [\gamma_{e_i}^{*}] = [\gamma_{\mathfrak{T}(e_i)}^{*}] }$, and $\upsilon_{\mathfrak{T}(e_i)} = -\upsilon_{e_i}$.
Then the pairing $\mathfrak{T}_{\mathbb{Z}^{k}}: e(\Omega_{\mathbb{Z}^{k}}) \rightarrow e(\Omega_{\mathbb{Z}^{k}})$
is given by
\begin{equation}
\label{z:k:cover:constr}
\mathfrak{T}_{\mathbb{Z}^{k}}(e_i(l_1, \dots, l_d)) = \mathfrak{T}(e_i)(l_1 +  \upsilon_{e_i} i([\gamma_1], [\gamma_{e_i}^{*}]), \dots, l_d + \upsilon_{e_i} i([\gamma_k],[\gamma_{e_i}^{*}])).
\end{equation}

\vspace{2mm}
\noindent
Luckily, this process simplifies in the $\mathbb{Z}^k$-covers we will encounter.\\
Having $\overline{\gamma} \subset \text{span}_{\mathbb{Z}} \{ h(e_i)\}_{i=1}^{d-r_{\mathrm{o}}+1}$ (that is, the original edges before we modified $\Omega$),
the formula shows that we don't need to add any edges to $\Omega$.
Also, we don't need to glue the polygons in $\Omega$, and the edges $\{ e_i \}_{i=d - (r_{\mathrm{o}} -1)}^{d}$, which are not part of the basis, are identified trivially: 
$\mathfrak{T}_{\mathbb{Z}^{k}}(e_i(\overline{l})) = \mathfrak{T}(e_i)(\overline{l})$.

Therefore, we can avoid modifying the configuration $\Omega$.
Also, we can work with additional marked points in $H_{1}(X, \Sigma \cup \Sigma_{\mathrm{m}}, \mathbb{Z})$, and still the formula holds the same.

\vspace{6mm}
\noindent
Let us mention two results about $\mathbb{Z}^k$-covers from \cite{HoW,DVH}.
We will use the homology group $H_1(X, \Sigma, \mathbb{Q})$ so that we can consider vector spaces.
Given $\overline{\gamma}$, let $V_{\overline{\gamma}}$ be the space $V_{\overline{\gamma}} \subset  H_1(X, \Sigma, \mathbb{Q}) = \text{span}_{\mathbb{Q}}(([\gamma_{1}], \ldots,[\gamma_{k}])$.

\begin{corollary}
$\widetilde{X}_{\overline{\gamma}_1} \cong \widetilde{X}_{\overline{\gamma}_2} \Longleftrightarrow  V_{\overline{\gamma}_1} = V_{\overline{\gamma}_2}$, where the first isomorphism is of translation surfaces.
\end{corollary}

\begin{corollary}
\label{cor:action:z:cover}
An affine automorphism $f \in \mathrm{Aff}(X)$ lifts to an affine automorphism
$\widetilde{f} \in \mathrm{Aff}(\widetilde{X})$ if and only if $f_*V_{\overline{\gamma}} = V_{\overline{\gamma}}$.
\end{corollary}

\vspace{5mm}
\noindent
Recall that we denoted by $\widetilde{\phi}_{t}^{\theta}(\widetilde{x})$ the lift of the flow $\phi_{t}^{\theta}(x)$ to $\widetilde{X}_{\overline{\gamma}}$.
We will say that $\overline{\gamma}$ fulfills the `no drift' condition if 
\begin{equation}
\omega([\gamma_1]) = \dots =  \omega([\gamma_k]) = 0.
\end{equation}

As mentioned in the introduction, the recurrence problem for a $\mathbb{Z}$-cover is well studied:
\begin{corollary}
\cite{HoW}
Let $\widetilde{X}_{[\gamma_1]}$ be a $\mathbb{Z}$-cover of $(X,\omega)$, defined by an element ${ [\gamma_{1}] \in H_{1}(X, \Sigma, \mathbb{Z}) }$. 
Assume that the flow $\phi_{t}^{\theta}(x)$ is ergodic.
Then the flow $\widetilde{\phi}_{t}^{\theta}(\widetilde{x})$ on $\widetilde{X}_{[\gamma_1]}$ is recurrent iff $[\gamma_1]$ has no drift.
\end{corollary}

Unfortunately, this does not hold in $\mathbb{Z}^k$-covers with $k \ge 2$. 
The no-drift condition and the ergodicity of $\phi_{t}^{\theta}(x)$ are not enough for ensuring the recurrence of the flow $\widetilde{\phi}_{t}^{\theta}(\widetilde{x})$, 
as shown by a counterexample of Delecroix \cite{De}.

\vspace{3mm}
\noindent
An important observation in \cite{HoW} is that a cylinder $C$ in $X$ lifts to isometric cylinders in
the cover $\widetilde{X}_{\overline{\gamma}}$ if and only if the core curve of the cylinder $m(C) \in H_1(X,\Sigma, \mathbb{Z})$ (the homology class of the geodesics foliating $C$) satisfies
${ i(m(C), [\gamma_i]) = 0 }$ for $i = 1, \dots, k$.

\subsection{$\mathbb{Z}^2$ torus covers}
\label{sec:z2:torus:cover}

We can describe tilings of the plane, as in Section~\ref{tilling}, in terms of $\mathbb{Z}^{2}$-covers.\\
Recalling the notions, we use the configuration $F_{\mathbb{Z}^2}$.
Then the tilling corresponds to the cover $\widetilde{X}_{\overline{\gamma}}$, where
\begin{equation}
\label{form:z2:torus:cover}
\overline{\gamma} = \left(  \sum_{j=1}^{d} a_{1}^{j} h(e_j),\sum_{j=1}^{d} a_{2}^{j} h(e_j)  \right).
\end{equation}

\noindent
Indeed, following formula~(\ref{z:k:cover:constr}) we verify:
${  \upsilon_{e_i} i(\sum_{j=1}^{d} a_{1}^{j} h(e_j), [\gamma_{e_{i}}^{*}])= a_{1}^{i}  }$,
and similarly for the second component.

\vspace{2mm}
\noindent
We will refer to such covers, corresponding to tilings, as {\em $\mathbb{Z}^2$ torus covers}.

\subsection{$\mathbb{Z}^k$-covers of a simple parking garage}
\label{cover:boundary}

We will apply $\mathbb{Z}^k$-covers also to simple translation surfaces with boundary $X_\mathrm{b}$.\\
As we will see, formula~(\ref{z:k:cover:constr}) will stay the same. 
One removes open disks from the interior of the polygons in $\Omega$ (and similarly in $\Omega_{\mathbb{Z}^k}$)
and carries the edge identification while ignoring the boundary.

This is where our restrictive simplicity definition from Section~\ref{sec:parking}, where the boundary does not intersect the edges, 
comes in handy.

\vspace{5mm}
\noindent
Having $X_\mathrm{b} = X\backslash \{\Delta_i\}_i$, and recalling the basis construction from Section~\ref{basis:homology},
we look at $H_1(X_\mathrm{b} \backslash \Sigma,\mathbb{Z})$ and $H_1(X \backslash \Sigma,\mathbb{Z})$.
For $\Delta_i$ removed from the interior of $O_{i'}$, we let $\gamma^*_{\Delta_i}$ be the loop going around the boundary of $\Delta_i$.
Using the basis of $H_1(X \backslash \Sigma,\mathbb{Z})$, and adding the classes $[\gamma^*_{\Delta_i}]$ generates $H_1(X_\mathrm{b} \backslash \Sigma,\mathbb{Z})$.
Since $i(h(e_i)),[\gamma^*_{\Delta_i}]) = 0$ for all the edges, for $\overline{\gamma} \subset \text{span}_{\mathbb{Z}} \{ h(e_i)\}_{i=1}^{d-r_{\mathrm{o}}+1}$
the validity of formula~(\ref{z:k:cover:constr}) is not affected.

More intuitively, we already allowed adding marked points $\Sigma_{\mathrm{m}}$ in the $\mathbb{Z}^k$-cover constructions.
One can think of shrinking the boundaries of the discs $\Delta_i$ to punctures by a weak deformation retraction, constructing the cover, and blowing the punctures back. 

\vspace{5mm}
\noindent
The construction holds the same for a simple parking garage, where the boundary components are just polygonal.
In the wind-tree unfolding, we will apply a $\mathbb{Z}^2$ torus cover to a simple parking garage.
Using the notions of Section~\ref{sec:z2:torus:cover}, $F$ will be a fundamental domain with the polygons removed from its interior.
By the results of this section, formula~(\ref{form:z2:torus:cover}) will still give the correct homology classes defining the cover.

\section{Unfolding of a periodic planar billiard}
\label{Unfolding:planar:billiard}

We will now present the unfolding procedure to a planar periodic billiard $\mathrm{WT}(\Lambda, \Omega)$,
and describe it as a $\mathbb{Z}^2$-cover. For the original wind-tree model, the process was described in \cite{DHL},
and now we will generalize it.

\vspace{3mm}
\noindent
We denoted $\Omega = \overline{F \backslash \{O_i\}_{i=1}^{r_{\mathrm{o}}}}$, where $F$ is some fundamental domain for $\Lambda$, and $O_i$ are the obstacles.
The obstacles $O_i$ are arbitrary closed domains with a connected boundary consisting of linear segments.
This possibly includes degenerate polygons (of zero area) and $2$-gons. 

In Section~\ref{tilling} we have seen how $F$ has a structure of translation surface.
Considering this torus structure on $F$, and letting the obstacles serve as boundary, $\Omega$ becomes a parking garage which we will denote by $P$.
We choose the domain $F$ so that the obstacles $\{O_i\}_{i}$ lie in its interior.
One can verify this is indeed possible, using cutting and gluing-in-parallel on the torus with boundary $F \backslash \{O_i\}_{i}$.
So $P$ is a simple parking garage.

We write $e(P)=e(F)\cup e(O)$, for the edges of $F$ and the edges of $\{O_i\}_{i=1}^{r_{\mathrm{o}}}$.
We will denote edges in $e(F)$ as $e_j \in e(F)$ and edges in $e(O)$ as $b_j \in e(O)$.
As in Section~\ref{tilling}, we pick $e_1,e_2  \in e^+(F)$, so that $\tau_1,\tau_2$ are linearly independent.
Now, each element of $\Lambda$ can be represented as $\tau = l_1\tau_1 + l_2\tau_2$, $ l_1, l_2 \in \mathbb{Z}$.

\vspace{3mm}
\noindent
The unfolding holds for rational configurations $\Omega$.
As in Section~\ref{sec:billiard:unfold}, we introduce $\xi_{\text{diff}}(\Omega) = \left\{\xi_{i}-\xi_{j}\right\}_{i, j} = \left\{  \frac{m_i}{n_i} \pi | n_i \in \mathbb{N}, m_i \in \mathbb{Z} \right\}_{i}$,
and the unfolding constant $n(P) := \text{l.c.m.}(\left\{   n_i   \right\}_i)$. 

We view $\mathrm{WT}(\Lambda, \Omega) = P^{\infty}$ as an infinite parking garage , with the obstacles forming the boundary.
Specifically, we assume $\mathrm{WT}(\Lambda, \Omega)$ is connected. Otherwise, the resulting surface can be described as a $\mathbb{Z}$-cover, or as a finite-type translation surface,
for which we know that recurrence holds.

\subsection{Unfolding procedure}

The unfolding procedure concerns four objects.

\vspace{3mm}
\noindent
The first object is $P^{\infty}$.\\
As we discussed, $P^{\infty}$ is an infinite parking garage.
We denote by $\partial P^{\infty}$ its boundary: $\partial P^{\infty} = e(O)_{\Lambda} = \{ b_j(l_1,l_2) \,\,|\,\, b_j \in e(O), l_1\tau_1 + l_2\tau_2 \in \Lambda \}$.
To model the wind-tree billiard, we will apply unfolding to $P^{\infty}$.

\vspace{3mm}
\noindent
The second object is the infinite translation surface $X^{\infty}$.\\
It results from the unfolding to $P^{\infty}$. For each $\rho_{i} \in D_n$ we reflect $P^{\infty}$ in a different copy of the plane,
$\rho_{1}P^{\infty}, \dots, \rho_{2n}P^{\infty}$, with $\rho_{1}P^{\infty} = P^{\infty}$. 
We identify each $\rho_i b_j(l_1,l_2) \in \rho_i e(O)_{\Lambda}$
with $\rho_i b_j(l_1,l_2) \sim \rho_{i} \rho_{b_j} b_j(l_1,l_2)$. 
This gives $X^{\infty}$. 
The billiard flow on $\mathrm{WT}(\Lambda, \Omega)$ lifts to a linear flow on $X^{\infty}$, and a linear flow on $X^{\infty}$ in turn projects to the billiard flow on $\mathrm{WT}(\Lambda, \Omega)$.

\vspace{3mm}
\noindent
The third object is the finite simple parking garage $P$.\\
Thinking of tillings of the plane as in Section~\ref{tilling}, we can construct $P^{\infty}$ as follows.
We take $e(F)_{\Lambda} = \{ e_j(l_1,l_2) | e_j \in e(F), l_1\tau_1 + l_2\tau_2 \in \Lambda \}$, 
and denote $e(\Omega) =  e(O)_{\Lambda} \cup e(F)_{\Lambda}$.
Then, as in formula~(\ref{form:lattice:identifications}), we identify 
$ e_{j}\left(l_{1}, l_{2}\right) \sim \mathfrak{T}_F\left(e_{j}\right)\left(l_{1}+a_{1}^{j}, l_{2}+a_{2}^{j}\right) $.
This gives $P^{\infty}$.
Following Sections~\ref{sec:z2:torus:cover} and~\ref{cover:boundary}, $P^{\infty}$ is a $\mathbb{Z}^2$ torus cover of $P$, using the cycles from formula~(\ref{form:z2:torus:cover}):
$\overline{\gamma}_P = \left(   \sum_{j=1}^{d} a_{1}^{j} h(e_j),\sum_{j=1}^{d} a_{2}^{j} h(e_j)   \right)$,
${ \overline{\gamma}_P \subset H_1(P,\Sigma(P),\mathbb{Z}) }$.

\vspace{5mm}
\noindent
The fourth object is the finite translation surface $X$.\\
It results from applying the unfolding procedure to $P$. For each $\rho_{i} \in D_n$ we reflect $P$ in a different copy of the plane,
$\rho_{1}P, \dots, \rho_{2n}P$, with $\rho_{1}P = P$. We identify each $\rho_i b_j \in \rho_i e(O)$
with $\rho_i b_j \sim \rho_{i} \rho_{b_{j}} b_{j}$. This gives $X$. 

\vspace{5mm}
\noindent
We now describe $X^{\infty}$ constructively. We work with the configuration $D_n (\Omega)_{\Lambda}$, 
and identify:
\begin{equation}
\label{form:x:infinity:edge:identifications}
\begin{aligned}
& \rho_{i}e_{j}(l_{1}, l_{2}) \sim \rho_{i}\mathfrak{T}_F(e_{j})(l_{1}+a_{1}^{j}, l_{2}+a_{2}^{j}), \\
& \rho_{i}b_{j}(l_{1}, l_{2}) \sim \rho_{i} \rho_{b_{j}} b_{j}(l_{1}, l_{2}).
\end{aligned}
\end{equation}
In detail, we write $X^{\infty} (\Omega,  \Lambda,  D_n)$ as a cartesian product
${ X^{\infty} = (\Omega  \times  \Lambda  \times D_n)/\sim }$ with the following equivalencies:
\begin{equation}
\begin{aligned}
&  X^{\infty}(x_{\alpha}, (l_1^{\alpha}, l_2^{\alpha}), \rho_{\alpha})=X^{\infty}(x_{\beta}, (l_1^{\beta},l_2^{\beta}), \rho_{\beta}) \,\, \Longleftrightarrow \\
&  \left\{\begin{array}
{c}{x_{\alpha}=x_{\beta},\, (l_1^{\alpha}, l_2^{\alpha})=(l_1^{\beta},l_2^{\beta}),\, \rho_{\alpha}=\rho_{\beta}     }\\ 
{x_{\alpha} \in e_j \in e(F),\, x_{\beta} = x_{\alpha} + \tau_{e_j},\, (l_1^{\beta},l_2^{\beta}) =   (l_1^{\alpha} +a_{1}^{j}, l_2^{\alpha}+a_{2}^{j})  ,\, \rho_{\alpha}=\rho_{\beta}    }\\
{x_{\alpha} \in b_j \in e(O),\,  x_{\beta} = x_{\alpha},\,  (l_1^{\alpha}, l_2^{\alpha})=(l_1^{\beta},l_2^{\beta}),\,   \rho_{\alpha}\rho_{b_j}\rho_{\alpha}^{-1} =  \rho_{\beta}   }
\end{array}\right.
\end{aligned}
\end{equation}
Then we have the projections
\begin{equation}
\label{form:wt:projections}
\begin{aligned}
&  \mathbb{P}_{P^{\infty}}: X^{\infty} (\Omega,  \Lambda, D_n) \rightarrow X^{\infty} (\Omega,  \Lambda, \, \cdot \,)  \quad\!\! \cong P^{\infty} ,\\
&  \mathbb{P}_{X} \quad\!\! : X^{\infty} (\Omega,  \Lambda, D_n) \rightarrow X^{\infty} (\Omega, \, \cdot \,, D_n) \cong X  ,\\
&  \mathbb{P}_{P} \quad\!: X^{\infty} (\Omega,  \Lambda, D_n) \rightarrow X^{\infty} (\Omega, \,  \cdot \, , \, \cdot \,) \quad\!\! \cong P  .
\end{aligned}
\end{equation}
Writing $D_n = \left\{\theta_{1}, \ldots, \theta_{n}, \rho \theta_{1}, \ldots, \rho \theta_{n}\right\}$, we will show:
\begin{proposition}
\label{prop:wt:cover}
$X^{\infty}$ is a $\mathbb{Z}^2$-cover of $X$, $X^{\infty} =  \widetilde{X}_{\overline{\gamma}}$, where  $\overline{\gamma} \subset H_1(X,\Sigma,\mathbb{Z})$ is given by
$$\overline{\gamma}=\left(
\begin{aligned}   
& [\gamma_1] =  \sum_{\theta_i \mathrm{\,rotation}}^{i = 1, \dots, n} \, \sum_{j=1}^{d} a_{1}^{j} {\theta_i}_* h\left(e_{j}\right)    -    \sum_{\rho\theta_i \mathrm{\,reflection}}^{i = 1, \dots, n} \, \sum_{j=1}^{d} a_{1}^{j} {\rho\theta_i}_* h\left(e_{j}\right)      \\
& [\gamma_2] =  \sum_{\theta_i \mathrm{\,rotation}}^{i = 1, \dots, n} \, \sum_{j=1}^{d} a_{2}^{j} {\theta_i}_* h\left(e_{j}\right)    -    \sum_{\rho\theta_i \mathrm{\,reflection}}^{i = 1, \dots, n} \, \sum_{j=1}^{d} a_{2}^{j} {\rho\theta_i}_* h\left(e_{j}\right)    
 \end{aligned}
 \right)$$
\end{proposition}

\vspace{6mm}
\noindent
Let us summarize:
\begin{enumerate}
\item $P \xrightarrow{D_n} X$ looking at $b_j \in e(\Omega)$.
\item $P \xrightarrow{\overline{\gamma}_P} P^{\infty}$ looking at $e_j \in e(F)$.
\item $P^{\infty} \xrightarrow{D_n} X^{\infty}$ looking at $b_j(l_1,l_2) \in e(\Omega_{\mathbb{Z}^2})$.
\item $X \xrightarrow{\overline{\gamma}} X^{\infty}$ looking at $\rho_i e_j \in D_n e(F)$.
\end{enumerate}


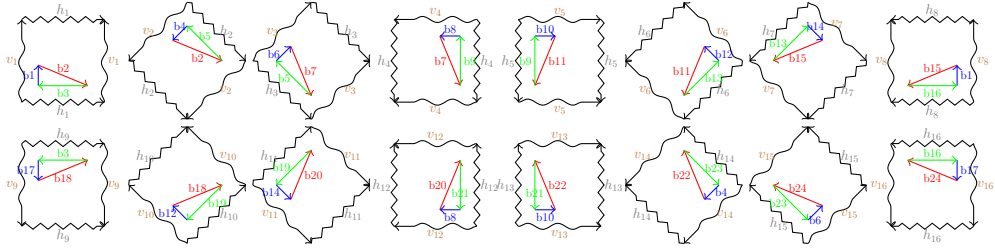
\begin{figure}[htbp]
\begin{tikzpicture}[scale=1.1,every node/.style = {opacity = 1, black, above left,scale = 0.37}] 

\newcommand\recVbaseX{0}; 
\newcommand\recVbaseY{0};
\coordinate (recV_base) at (\recVbaseX,\recVbaseY); 

\newcommand\recEl{1} 
\coordinate (recE) at (\recEl,\recEl);

\newcommand\trVax{0.2} \newcommand\trVay{0.2} 
\newcommand\trVbx{0.2} \newcommand\trVby{0.448528}
\newcommand\trVcx{0.8} \newcommand\trVcy{0.2}
\newcommand\trRfVax{0.2} \newcommand\trRfVay{0.8} 
\newcommand\trRfVbx{0.2} \newcommand\trRfVby{0.551472}
\newcommand\trRfVcx{0.8} \newcommand\trRfVcy{0.8}

\newcommand\recDl{1.5*\recEl} 

\foreach \row / \ref /  \rot  / \index / \trEaN/\trEaO/\trEbN/\trEbO/\trEcN/\trEcO   in {0/0/0/0  /b1/left/b2/above/b3/below/,   0/1/-45/1  /b4/above/b2/below/b5/above/,   0/0/-45/2  /b6/left/b7/right/b5/left/,   0/1/-90/3  /b8/above/b7/left/b9/right/,   0/0/-90/4  /b10/above/b11/right/b9/left/,   0/1/-135/5  /b12/right/b11/left/b13/right/,   0/0/-135/6  /b14/above/b15/below/b13/left/,   0/1/-180/7 /b1/right/b15/above/b16/below/,
					         1/0/0/0  /b17/left/b18/below/b3/above/,  1/1/-45/1   /b12/left/b18/above/b19/right/,  1/0/-45/2   /b14/left/b20/right/b19/left/,  1/1/-90/3   /b8/below/b20/left/b21/below/,  1/0/-90/4   /b10/below/b22/right/b21/below/,  1/1/-135/5   /b4/right/b22/left/b23/right/,  1/0/-135/6  /b6/below/b24/above/b23/left/,  1/1/-180/7  /b17/right/b24/below/b16/above/	} 
{
	\coordinate (recV) at  ($(\recVbaseX+\recDl*\index,\recVbaseY - \row*\recDl)$);
	\newcommand\rotO{\rot * (-1)^\row}
	
	\coordinate (recC) at ($(recV)+0.5*(recE)$); 
	\pgfmathparse{(\row == 0 || \row == 1) && \ref==0 && \rot>-45 ?int(1):int(0)}
	\ifnum\pgfmathresult>0\relax
		\newcommand\OrHb{below} \fi
	\pgfmathparse{(\row == 0 || \row == 1) && \ref==1 && \rot <= -180 ?int(1):int(0)}
	\ifnum\pgfmathresult>0\relax
		\newcommand\OrHb{above} \fi
	\pgfmathparse{\row == 0 &&( (\ref==0 && \rot <= -45) || (\ref==1 && \rot >-180) ) ?int(1):int(0)}
	\ifnum\pgfmathresult>0\relax
		\newcommand\OrHb{left}	\fi
	\pgfmathparse{\row == 1 && ( (\ref==0 && \rot <= -45) || (\ref==1 && \rot > -180) )  ?int(1):int(0)}
	\ifnum\pgfmathresult>0\relax
		\newcommand\OrHb{right} \fi
	\pgfmathparse{(\row == 0 || \row == 1) && \ref==0 && \rot>-45 ?int(1):int(0)}
	\ifnum\pgfmathresult>0\relax
		\newcommand\OrHt{above} \fi
	\pgfmathparse{(\row == 0 && \ref==1 && \rot <= -180) || (\row == 1 && ( (\ref == 0 && \rot <= -135) || (\ref == 1 && \rot <= -180)  )) ?int(1):int(0)}
	\ifnum\pgfmathresult>0\relax
		\newcommand\OrHt{below} \fi
	\pgfmathparse{\row == 0 &&( (\ref==0 && \rot <= -45) || (\ref==1 && \rot > -180  )) ?int(1):int(0)}
	\ifnum\pgfmathresult>0\relax
		\newcommand\OrHt{right} \fi
	\pgfmathparse{\row == 1 && ( (\ref == 0 && \rot <= -45 && \rot >-135) || (\ref == 1 && \rot > -180) ?int(1):int(0)}
	\ifnum\pgfmathresult>0\relax
		\newcommand\OrHt{left} \fi
	\pgfmathparse{(\row == 0 || \row == 1) && ((\ref==0 &&\rot>-90)||(\ref==1 && \rot > -90)) ?int(1):int(0)}
	\ifnum\pgfmathresult>0\relax
		\newcommand\OrHl{left} \fi
	\pgfmathparse{(\row == 0 && \ref==1 &&\rot<=-135) || (\row == 1 && ( (\ref == 0 && \rot <= -135) || (\ref == 1 && \rot <= -135)    )) ?int(1):int(0)}
	\ifnum\pgfmathresult>0\relax
		\newcommand\OrHl{right} \fi
	\pgfmathparse{\row == 0 && ((\ref==0 &&\rot<=-90) || (\ref == 1 && \rot <= -90 && \rot > -135)) ?int(1):int(0)}
	\ifnum\pgfmathresult>0\relax
		\newcommand\OrHl{above} \fi	
	\pgfmathparse{\row == 1 && ((\ref==0 && \rot<=-90 && \rot > -135) || (\ref == 1 && \rot <= -90 && \rot > -135)) ?int(1):int(0)}
	\ifnum\pgfmathresult>0\relax
		\newcommand\OrHl{below} \fi	
	\pgfmathparse{(\row == 0 || \row == 1) && ((\ref==0 &&\rot>-90)||(\ref==1 && \rot > -90)) ?int(1):int(0)}
	\ifnum\pgfmathresult>0\relax
		\newcommand\OrHr{right} \fi
	\pgfmathparse{(\row == 0 || \row == 1) && ((\ref==0 &&\rot<=-135)||(\ref==1 && \rot <= -135)) ?int(1):int(0)}
	\ifnum\pgfmathresult>0\relax
		\newcommand\OrHr{left} \fi
	\pgfmathparse{\row == 0 && ((\ref==0 &&\rot<=-90&&\rot>-135) || (\ref == 1 && \rot <= -90 && \rot > -135)) ?int(1):int(0)}
	\ifnum\pgfmathresult>0\relax
		\newcommand\OrHr{below} \fi
		\pgfmathparse{\row == 1 && ((\ref==0 &&\rot<=-90&&\rot>-135) || (\ref == 1 && \rot <= -90 && \rot > -135)) ?int(1):int(0)}
	\ifnum\pgfmathresult>0\relax
		\newcommand\OrHr{above} \fi
	%
	
	\def \HnIdx {\pgfmathparse{int(8*\row + (\index + 1))}\pgfmathresult}
	\newcommand\HoneRho{$h_{\HnIdx}$}
	\newcommand\HtwoRho{$v_{\HnIdx}$}

	\draw (recV) [rotate around = {\rotO: (recC)}] +(0,0) -- +(\recEl,0) [->,decorate,decoration={zigzag,amplitude=0.4mm,segment length=2mm,post length=1mm}] node[midway,\OrHb]{\Large \textcolor{gray} {\HoneRho}};
	\draw (recV) [rotate around = {\rotO: (recC)}] +(0,\recEl) -- +(\recEl,\recEl) [->,decorate,decoration={zigzag,amplitude=0.4mm,segment length=2mm,post length=1mm}] node[midway,\OrHt]{\Large \textcolor{gray} {\HoneRho}};
	\ifnum \ref=\row
		\draw (recV) [rotate around = {\rotO: (recC)}] +(0,0) -- +(0,\recEl) [->,decorate,decoration={snake,amplitude=0.4mm,segment length=4mm,post length=1mm}] node[midway,\OrHl]{\Large \textcolor{brown} {\HtwoRho}};
		\draw (recV) [rotate around = {\rotO: (recC)}] +(\recEl,0) -- +(\recEl,\recEl) [->,decorate,decoration={snake,amplitude=0.4mm,segment length=4mm,post length=1mm}] node[midway,\OrHr]{\Large \textcolor{brown} {\HtwoRho}};
	\else
		\draw (recV) [rotate around = {\rotO: (recC)}] +(0,\recEl) -- +(0,0) [->,decorate,decoration={snake,amplitude=0.4mm,segment length=4mm,post length=1mm}] node[midway,\OrHl]{\Large \textcolor{brown} {\HtwoRho}};
		\draw (recV) [rotate around = {\rotO: (recC)}] +(\recEl,\recEl) -- +(\recEl,0) [->,decorate,decoration={snake,amplitude=0.4mm,segment length=4mm,post length=1mm}] node[midway,\OrHr]{\Large \textcolor{brown} {\HtwoRho}};
	\fi
	
	\ifnum \ref=\row
		\draw (recV)[rotate around = {\rotO: (recC)}] +(\trVax,\trVay) -- +(\trVbx,\trVby) [->,blue] node[midway,\trEaO]{\large \textcolor{blue} {\trEaN}};
		\draw (recV)[rotate around = {\rotO: (recC)}] +(\trVbx,\trVby) -- +(\trVcx,\trVcy) [->,red]   node[midway,\trEbO]{\large \textcolor{red} {\trEbN}};
		\draw (recV)[rotate around = {\rotO: (recC)}] +(\trVcx,\trVcy) -- +(\trVax,\trVay) [->,green] node[midway,\trEcO]{\large \textcolor{green} {\trEcN}};
	\else
		\draw (recV)[rotate around = {\rotO: (recC)}] +(\trRfVax,\trRfVay) -- +(\trRfVbx,\trRfVby) [->,blue] node[midway,\trEaO]{\large \textcolor{blue} {\trEaN}};
		\draw (recV)[rotate around = {\rotO: (recC)}] +(\trRfVbx,\trRfVby) -- +(\trRfVcx,\trRfVcy) [->,red]  node[midway,\trEbO]{\large \textcolor{red} {\trEbN}};
		\draw (recV)[rotate around = {\rotO: (recC)}] +(\trRfVcx,\trRfVcy) -- +(\trRfVax,\trRfVay)  [->,green]  node[midway,\trEcO]{\large \textcolor{green} {\trEcN}};
	\fi
}

\end{tikzpicture}
\caption{An example of $X$ for $\Omega$ consisting of a triangle with angles $\frac{\pi}{2},\frac{3\pi}{8},\frac{\pi}{8}$.}
\end{figure}

\begin{figure}[htbp]
\begin{tikzpicture}[scale=0.8,every node/.style = {opacity = 1, black, above left,scale = 0.4}] 

\newcommand\recVbaseX{0}; 
\newcommand\recVbaseY{0};
\coordinate (recV_base) at (\recVbaseX,\recVbaseY); 

\newcommand\recEl{1} 
\coordinate (recE) at (\recEl,\recEl);

\newcommand\trVax{0.2} \newcommand\trVay{0.2} 
\newcommand\trVbx{0.2} \newcommand\trVby{0.448528}
\newcommand\trVcx{0.8} \newcommand\trVcy{0.2}
\newcommand\trRfVax{0.2} \newcommand\trRfVay{0.8} 
\newcommand\trRfVbx{0.2} \newcommand\trRfVby{0.551472}
\newcommand\trRfVcx{0.8} \newcommand\trRfVcy{0.8}

\newcommand\recDl{2*\recEl} 

\foreach \row in {0,1}
{
\foreach \ref  \rot  / \index in {0/0/0 , 1/-45/1, 0/-45/2, 1/-90/3, 0/-90/4, 1/-135/5, 0/-135/6, 1/-180/7} 
{
	\coordinate (recV) at  ($(\recVbaseX+\recDl*\index,\recVbaseY - \row*\recDl)$);
	\newcommand\rotO{\rot * (-1)^\row}
	
	\coordinate (recC) at ($(recV)+0.5*(recE)$); 
	\pgfmathparse{(\row == 0 || \row == 1) && \ref==0 && \rot>-45 ?int(1):int(0)}
	\ifnum\pgfmathresult>0\relax
		\newcommand\OrHb{below} \fi
	\pgfmathparse{(\row == 0 || \row == 1) && \ref==1 && \rot <= -180 ?int(1):int(0)}
	\ifnum\pgfmathresult>0\relax
		\newcommand\OrHb{above} \fi
	\pgfmathparse{\row == 0 &&( (\ref==0 && \rot <= -45) || (\ref==1 && \rot >-180) ) ?int(1):int(0)}
	\ifnum\pgfmathresult>0\relax
		\newcommand\OrHb{left}	\fi
	\pgfmathparse{\row == 1 && ( (\ref==0 && \rot <= -45) || (\ref==1 && \rot > -180) )  ?int(1):int(0)}
	\ifnum\pgfmathresult>0\relax
		\newcommand\OrHb{right} \fi
	\pgfmathparse{(\row == 0 || \row == 1) && \ref==0 && \rot>-45 ?int(1):int(0)}
	\ifnum\pgfmathresult>0\relax
		\newcommand\OrHt{above} \fi
	\pgfmathparse{(\row == 0 && \ref==1 && \rot <= -180) || (\row == 1 && ( (\ref == 0 && \rot <= -135) || (\ref == 1 && \rot <= -180)  )) ?int(1):int(0)}
	\ifnum\pgfmathresult>0\relax
		\newcommand\OrHt{below} \fi
	\pgfmathparse{\row == 0 &&( (\ref==0 && \rot <= -45) || (\ref==1 && \rot > -180  )) ?int(1):int(0)}
	\ifnum\pgfmathresult>0\relax
		\newcommand\OrHt{right} \fi
	\pgfmathparse{\row == 1 && ( (\ref == 0 && \rot <= -45 && \rot >-135) || (\ref == 1 && \rot > -180) ?int(1):int(0)}
	\ifnum\pgfmathresult>0\relax
		\newcommand\OrHt{left} \fi
	\pgfmathparse{(\row == 0 || \row == 1) && ((\ref==0 &&\rot>-90)||(\ref==1 && \rot > -90)) ?int(1):int(0)}
	\ifnum\pgfmathresult>0\relax
		\newcommand\OrHl{left} \fi
	\pgfmathparse{(\row == 0 && \ref==1 &&\rot<=-135) || (\row == 1 && ( (\ref == 0 && \rot <= -135) || (\ref == 1 && \rot <= -135)    )) ?int(1):int(0)}
	\ifnum\pgfmathresult>0\relax
		\newcommand\OrHl{right} \fi
	\pgfmathparse{\row == 0 && ((\ref==0 &&\rot<=-90) || (\ref == 1 && \rot <= -90 && \rot > -135)) ?int(1):int(0)}
	\ifnum\pgfmathresult>0\relax
		\newcommand\OrHl{above} \fi	
	\pgfmathparse{\row == 1 && ((\ref==0 && \rot<=-90 && \rot > -135) || (\ref == 1 && \rot <= -90 && \rot > -135)) ?int(1):int(0)}
	\ifnum\pgfmathresult>0\relax
		\newcommand\OrHl{below} \fi	
	\pgfmathparse{(\row == 0 || \row == 1) && ((\ref==0 &&\rot>-90)||(\ref==1 && \rot > -90)) ?int(1):int(0)}
	\ifnum\pgfmathresult>0\relax
		\newcommand\OrHr{right} \fi
	\pgfmathparse{(\row == 0 || \row == 1) && ((\ref==0 &&\rot<=-135)||(\ref==1 && \rot <= -135)) ?int(1):int(0)}
	\ifnum\pgfmathresult>0\relax
		\newcommand\OrHr{left} \fi
	\pgfmathparse{\row == 0 && ((\ref==0 &&\rot<=-90&&\rot>-135) || (\ref == 1 && \rot <= -90 && \rot > -135)) ?int(1):int(0)}
	\ifnum\pgfmathresult>0\relax
		\newcommand\OrHr{below} \fi
		\pgfmathparse{\row == 1 && ((\ref==0 &&\rot<=-90&&\rot>-135) || (\ref == 1 && \rot <= -90 && \rot > -135)) ?int(1):int(0)}
	\ifnum\pgfmathresult>0\relax
		\newcommand\OrHr{above} \fi
	%
	
	\ifnum \rot = 0
		\ifnum \row = 0
			\newcommand\HoneRho{$h_1$}
			\newcommand\HtwoRho{$h_2$}
		\else
			\newcommand\HoneRho{$\rho h_1$}
			\newcommand\HtwoRho{$\rho h_2$}
		\fi
	\else 
		\def \RhoDnIdx {\pgfmathparse{int(4*\row + (\index + 1)/2)}\pgfmathresult}
		\ifnum \ref = \row
			\newcommand\HoneRho{$h_1^{\theta_{\RhoDnIdx}}$}
			\newcommand\HtwoRho{$h_2^{\theta_{\RhoDnIdx}}$}
		\else
			\newcommand\HoneRho{$h_1^{\rho\theta_{\RhoDnIdx}}$}
			\newcommand\HtwoRho{$h_2^{\rho\theta_{\RhoDnIdx}}$}
		\fi
		\relax
	\fi

	\draw (recV) [rotate around = {\rotO: (recC)}] +(0,0) -- +(\recEl,0) [->,thick,blue,decorate,decoration={zigzag,amplitude=0.4mm,segment length=2mm,post length=1mm}] node[midway,\OrHb]{\LARGE \textcolor{blue} {\HoneRho}};
	\draw (recV) [rotate around = {\rotO: (recC)}] +(0,\recEl) -- +(\recEl,\recEl) [->,thick,blue,decorate,decoration={zigzag,amplitude=0.4mm,segment length=2mm,post length=1mm}] node[midway,\OrHt]{\LARGE \textcolor{blue} {\HoneRho}};
	\ifnum \ref=\row
		\draw (recV) [rotate around = {\rotO: (recC)}] +(0,0) -- +(0,\recEl) [->,thick,red,decorate,decoration={snake,amplitude=0.4mm,segment length=4mm,post length=1mm}] node[midway,\OrHl]{\LARGE \textcolor{red} {\HtwoRho}};
		\draw (recV) [rotate around = {\rotO: (recC)}] +(\recEl,0) -- +(\recEl,\recEl) [->,thick,red,decorate,decoration={snake,amplitude=0.4mm,segment length=4mm,post length=1mm}] node[midway,\OrHr]{\LARGE \textcolor{red} {\HtwoRho}};
	\else
		\draw (recV) [rotate around = {\rotO: (recC)}] +(0,\recEl) -- +(0,0) [->,thick,red,decorate,decoration={snake,amplitude=0.4mm,segment length=4mm,post length=1mm}] node[midway,\OrHl]{\LARGE \textcolor{red} {\HtwoRho}};
		\draw (recV) [rotate around = {\rotO: (recC)}] +(\recEl,\recEl) -- +(\recEl,0) [->,thick,red,decorate,decoration={snake,amplitude=0.4mm,segment length=4mm,post length=1mm}] node[midway,\OrHr]{\LARGE \textcolor{red} {\HtwoRho}};
	\fi
	
	\ifnum \ref=\row
		\draw (recV)[rotate around = {\rotO: (recC)}] +(\trVax,\trVay) -- +(\trVbx,\trVby) -- +(\trVcx,\trVcy) -- +(\trVax,\trVay);
	\else
		\draw (recV)[rotate around = {\rotO: (recC)}] +(\trRfVax,\trRfVay) -- +(\trRfVbx,\trRfVby) -- +(\trRfVcx,\trRfVcy) -- +(\trRfVax,\trRfVay);
	\fi
}
}

\end{tikzpicture}
\caption{Example for the cover cycles.}
\end{figure}
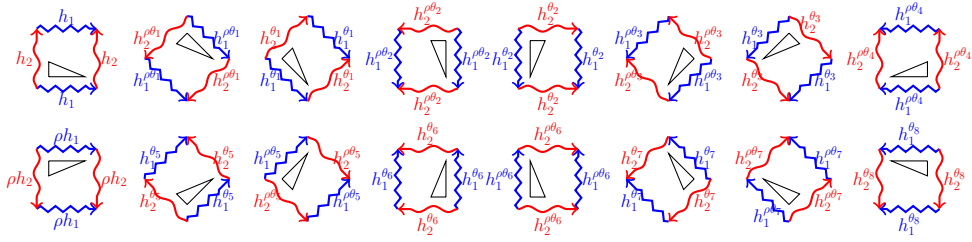

\subsection{Proof of the construction}

As in Subsection~\ref{sec:billiard:action}, $D_n$ acts on $X$ and on $X^{\infty}$ by affine automorphisms.
By formula~(\ref{form:wt:projections}),  the action of $D_n$ on $X^{\infty}$ is the lift of the action of $D_n$ on $X$: 
$  {\rho_i}^{(X)} \circ \mathbb{P}_{X}  =  \mathbb{P}_{X} \circ  {\rho_i}^{(X^{\infty})} $, for each $\rho_i \in D_n$.
In view of Corollary~\ref{cor:action:z:cover}, an initial guess for the cycles $\overline{\gamma}$ giving $X^{\infty}$ as a $\mathbb{Z}^2$-cover is
${D_n}_* \overline{\gamma}_P = \left\{  {D_n}_* {\gamma_l}_P = \sum_{i=1}^{n} \sum_{j=1}^{d} a_{l}^{j} \rho_{i *} h\left(e_{j}\right)                    \right\}$,
$l \in \{1,2\}$.
This will almost be the case. Indeed, $\rho_i P^{\infty}$ is a $\mathbb{Z}^2$-cover of $\rho_i P$, and the cover correlates to the subspace generated by $\rho_i \overline{\gamma}_P$.
But the delicate point is Corollary~\ref{cor:billiard:sign}, which shows that the intersection sign is affected by the $D_n$-action.
\begin{proof}[Proof of Proposition \ref{prop:wt:cover}]
We follow formula~(\ref{z:k:cover:constr}). 
Concerning the discussion there, on the assumptions on $\Omega$, we note: 
\begin{enumerate}
\item $\overline{\gamma} \subset { \text{span}_{\mathbb{Z}} \{ h(e_i | e_i \in e(F) )\}_{i=1}^{d} }$.
\item One can glue in parallel the parking garages $D_n P$ of $\Omega$ along the edges $D_n e(O)$, without changing $D_n e(F)$.
\end{enumerate}
Thus the formula holds, without any needed modifications to $\Omega$.

Now, by formula~(\ref{form:x:infinity:edge:identifications}), we should have $\rho_{i}e_{j}(l_{1}, l_{2}) \sim \rho_{i}\mathfrak{T}_F(e_{j}) \left( l_{1}+a_{1}^{j}, l_{2}+a_{2}^{j} \right)$.
So we should verify that $\rho_{i}\mathfrak{T}_F(e_{j})(l_{1}+a_{1}^{j}) = \mathfrak{T}_X(\rho_i e_j)(l_1 +  \upsilon_{\rho_i e_j} i([\gamma_1],[\gamma^*_{\rho_i e_j}]))$, and similarly for $[\gamma_2]$.
We have $\rho_{i}\mathfrak{T}_F(e_{j}) = \mathfrak{T}_X(\rho_i e_j)$ by the definition of $X$.
Note that if $\rho\theta_i$ is a reflection, then $\rho\theta_{i*}[\gamma^*_{e_j}] = - [\gamma^*_{\rho\theta_i e_j}]$ by Corollary~\ref{cor:billiard:sign}.
We see that:
$\upsilon_{\rho_i e_j} i([\gamma_1], [\gamma^*_{\rho_i e_j}]) = \\
\upsilon_{\rho_i e_j} i(
\sum_{\theta_{i'} \text { rotation }}^{i'=1, \ldots, n} \sum_{j'=1}^{d} a_{1}^{j'} \theta_{i' *} h(e_{j'})-\sum_{\rho \theta_{i'} \text { reflection }}^{i'=1, \ldots, n} \sum_{j'=1}^{d} a_{1}^{j'} \rho \theta_{i' *} h(e_{j'}) \, ,  \,
[\gamma^*_{\rho_i e_j}] )  =  \\
\{\begin{array}{ll}{(+1) \cdot i(a_{1}^{j}  \theta_{i*} h(e_{j}), \theta_{i*} [\gamma^*_{e_{j}}]) = a_{1}^{j} i(\theta_{i*} h(e_{j}),\theta_{i*}[\gamma^*_{e_{j}}] )  = a_{1}^{j},} & {\rho_i = \;\;\theta_i ,} \\ 
{(-1) \cdot i(- a_{1}^{j} \rho\theta_{i*} h (e_{j}), -\rho \theta_{i*}[\gamma^*_{e_{j}}]) = -a_{1}^{j} i(\rho\theta_{i*} h (e_{j}),\rho \theta_{i*}[\gamma^*_{e_{j}}] ) = a_{1}^{j},} & {\rho_i = \rho\theta_{i-n} .}\end{array}  $
\end{proof}

\section{Recurrence criterion}
\label{sec:rec:criterion}

We next present the recurrence criterion used in \cite{AH}.\\
Let $X \in \mathcal{H}_{1}(\kappa)$ be a translation surface, and let $\widetilde{X}_{\overline{\gamma}}$ be a $\mathbb{Z}^k$-cover of $X$ defined by $\overline{\gamma} \subset  H_1(X,\Sigma,\mathbb{Z})$.
We denoted by $\phi_{t}(x)$ the vertical flow on $X$ starting at $x$, and by $\widetilde{\phi}_{t}(\widetilde{x})$ its lift to $\widetilde{X}_{\overline{\gamma}}$.
Let $\mathcal{M}$ be the $\mathrm{SL}(2,\mathbb{R})$-orbit closure of $X$ in $\mathcal{H}_{1}(\kappa)$. 
Assume there is a symplectic splitting of the (extended real) Hodge bundle $H_{g}^{1} = V(\mathcal{M}) \oplus V^{\perp}(\mathcal{M})$ invariant for the KZ-cocycle on $\mathcal{M}$
(as discussed in Corollary~\ref{cor:quad:hodge:split}).
Assume that $\overline{\gamma} \in V_X \cap H_1(X,\Sigma,\mathbb{Z})$.
A cylinder $C$ in $X$ will be called {\em good} if the homology class of its core curve $m(C)$ belongs to $V^{\perp}_X$.

\begin{theorem}
\cite{AH}
\label{AH:rec:criterion}
Assume that the `no drift' condition holds: $\omega([\gamma_i] \in \overline{\gamma}) = 0$.
Let $(Y,\omega_Y) \in \mathcal{M}$ be a surface with a vertical good cylinder $C$, $m(C) \in V^{\perp}_Y$.
Assume that the orbit $g_t(X,\omega)$ accumulates at $(Y,\omega_Y)$.
If the vertical flow $\phi_{t}(x)$ on $X$ is ergodic, then the vertical flow $\widetilde{\phi}_{t}(\widetilde{x})$ on $\widetilde{X}_{\overline{\gamma}}$ is recurrent.
\end{theorem}

\vspace{3mm}
\noindent
The theorem is an adaptation of Masur's criterion~\ref{thm:masur:unique:ergod}, together with the observation in Section~\ref{sec:z:d:cover} that cylinders lift to isometric cylinders in the cover
if ${ i(m(C), [\gamma_i]) = 0}$, $ i = 1, \dots, k$. 

Following that, another recurrence criterion was shown in \cite{FH,Pa}, which is a direct application of the Avila-Hubert criterion~\ref{AH:rec:criterion} together with the Chaika-Eskin result~\ref{Chaika-Eskin}:

\begin{theorem}
\label{rec:criterion}
Assume that the `no drift' condition holds, $\omega([\gamma_i] \in \overline{\gamma}) = 0$.
Let $(Y,\omega_Y) \in \mathcal{M}(X)$ be a surface with a good cylinder $C$, not necessarily vertical, such that $m(C) \in V^{\perp}_Y$.
Then for every Birkhoff generic direction $\theta \in [0,2\pi)$ for $X$, the flow $\widetilde{\phi}_{t}^{\theta}$ is recurrent.
Consequently the flow on $\widetilde{X}_{\overline{\gamma}}$ is recurrent in almost every direction.
\end{theorem}

\noindent
We will use the latter criterion in the proof of the main theorem. 
We will do so with $Y = X$.

\section{Proof of the main theorem}

We turn to the proof of the main result, Theorem~\ref{thm:main}.\\
Using the unfolding procedure from Section~\ref{Unfolding:planar:billiard} and the recurrence criterion~\ref{rec:criterion}, we can rephrase the main theorem as follows.

\vspace{5mm}
\noindent
Given a wind-tree model $\mathrm{WT}(\Lambda, \Omega)$ with rational configuration $\Omega$, we construct a new rational configuration $\widehat{\Omega}$
by adding obstacles to $\Omega$. We look at $\mathrm{WT}(\Lambda, \widehat{\Omega})$ and want to show that upon fixing a starting point $p \in WT(\Lambda, \widehat{\Omega})$, 
the billiard flow $\mathcal{F}_{t}^{\theta}(p)$ on $\mathrm{WT}(\Lambda, \widehat{\Omega})$ is recurrent in almost every direction $\theta \in [0,2\pi)$.

As in Section~\ref{Unfolding:planar:billiard}, we can model the billiard flow on $\mathrm{WT}(\Lambda, \widehat{\Omega})$
using the linear flow on an infinite translation surface $X^{\infty}$. 
By Proposition~\ref{prop:wt:cover}, $X^{\infty} = \widetilde{X}_{\overline{\gamma}}$ is a $\mathbb{Z}^2$-cover of a (finite type) translation surface $X$,
$\overline{\gamma} \subset H_1(X,\Sigma,\mathbb{Z})$. 
Let $\mathcal{M}$ denote the $\mathrm{SL}(2,\mathbb{R})$-orbit closure of $(X,\omega)$ in $\mathcal{H}_{1}(\kappa)$.
We show:

\begin{theorem}
\label{thm:main:re}
\vspace{1mm}
\noindent
$\widehat{\Omega}$ can be constructed so that the following hold.
\begin{enumerate}
\item $X$ has an involution $\iota \in \text{Aff}(X,\Sigma)$.
\end{enumerate}
 By Corollary~\ref{cor:quad:hodge:split}, this gives a splitting of the Hodge bundle ${ H_{g}^{1}(\mathcal{M}) = V(\mathcal{M}) \oplus V^{\perp}(\mathcal{M}) }$,
 invariant for the KZ-cocycle. The fibers are $H_{g}^{1} (X) = H_1^{\iota+}(X,\Sigma,\mathbb{R}) \oplus H_1^{\iota-}(X,\Sigma,\mathbb{R})$.
\begin{enumerate}
\setcounter{enumi}{1}
\item The `no drift' condition is satisfied: $\omega([\gamma_i] \in \overline{\gamma}) = 0$.
\item $ \overline{\gamma} \subset  H_1^{\iota+}(X,\Sigma,\mathbb{R})  $.
\item There exists a good cylinder $C$ in $X$ with $m(C) \in  H_1^{\iota-}(X,\Sigma,\mathbb{R})$.
\end{enumerate}
\end{theorem}
Using Criterion~\ref{rec:criterion}, we conclude that the linear flow on $\widetilde{X}_{\overline{\gamma}}$, and thus also the billiard flow on $\mathrm{WT}(\Lambda, \widehat{\Omega})$, are recurrent in almost every direction.

\vspace{5mm}
\noindent
We describe the construction of $\widehat{\Omega}$ (see also the following figure).\\
We take two L-shaped polygons, $L_1,L_2$, with corners $s_1,s_2$ of angle $\frac{1}{2}\pi$.
We will put $\widehat{\Omega} = \Omega \backslash \{L_1,L_2\}$.
We want to keep $\widehat{\Omega}$ rational, so we take some $\xi \in \xi_{\text{diff}}(\Omega)$, and rotate $L_1$ by $\xi$.
We rotate $L_2$  by $\xi + \pi$.
We take $L_1,L_2$ small enough and place them close to each other,
so that we can draw a straight line between $s_1$ to $s_2$ inside $\widehat{\Omega}$, $\mathcal{F}(s_1,s_2)$.

\vspace{5mm}
\noindent
There are other constructions that will work as well,
namely any two rational polygons with corners of angle $\frac{1}{2t}\pi$ that can be connected by a straight line in $\widehat{\Omega}$.
We see that $\xi_{\text{diff}}(\widehat{\Omega}) = \xi_{\text{diff}}(\Omega) \cup \{ \frac{\pi}{2} \}$.
So the unfolding constant $n(\widehat{\Omega})$ is even. 
We note that for $n$ even, the dihedral group $D_n$ contains a rotation by $\pi$, which we denote by $\iota$.
As in Section~\ref{sec:billiard:action}, $\iota$ is an involution in $\text{Aff}(X,\Sigma)$. 
This gives \ref{thm:main:re}(1).

Unfolding $P(\widehat{\Omega})$ to $X$, we note that $s_1,s_2$ are regular singularities of $X$.
Indeed $s_1$ is identified with $\rho_ {\xi} (s_1),\iota (s_1)$ and $\rho_ {\xi} \circ \iota (s_1)$,
to give a cone angle of $4 \cdot \frac{\pi}{2} = 2\pi$ (and similarly for $s_2$).
We will consider the geodesic flow going through them.

\vspace{5mm}
\noindent
For the cylinder $C$, we construct a closed geodesic $\gamma_C$ in $X$, which will be the core curve of $C$ (also see Figure~\ref{good:cylinder:fig}).

Go along $\mathcal{F}(s_1,s_2)$. 
Upon hitting the regular singularity $s_2$, the geodesic flow continues through $\iota(s_2)$ and goes along $\iota(\mathcal{F}(s_1,s_2))$.
Upon hitting $\iota(s_1)$ we are back at the start at $s_1$

\begin{figure}
\centering

\tikzset{every picture/.style={line width=0.75pt}} 

\begin{tikzpicture}[x=0.75pt,y=0.75pt,yscale=-1,xscale=1]

\draw  [fill={rgb, 255:red, 23; green, 5; blue, 5 }  ,fill opacity=1 ] (231,41) -- (291,100.33) -- (231,100.33) -- cycle ;
\draw   (221,32.33) -- (301,32.33) -- (301,109.33) -- (221,109.33) -- cycle ;
\draw  [line width=3]  (278,44) -- (288,44) -- (288,55) ;
\draw  [line width=3]  (279,65) -- (269,65) -- (269,54) ;
\draw [color={rgb, 255:red, 255; green, 9; blue, 0 }  ,draw opacity=1 ][line width=1.5]    (272,61.75) -- (283.32,49.22) ;
\draw [shift={(286,46.25)}, rotate = 492.09] [fill={rgb, 255:red, 255; green, 9; blue, 0 }  ,fill opacity=1 ][line width=0.08]  [draw opacity=0] (9.29,-4.46) -- (0,0) -- (9.29,4.46) -- cycle    ;
\draw  [fill={rgb, 255:red, 23; green, 5; blue, 5 }  ,fill opacity=1 ] (401.69,100.82) -- (341.5,39.33) -- (401.69,39.33) -- cycle ;
\draw   (411,109.81) -- (329.5,109.81) -- (329.5,30) -- (411,30) -- cycle ;
\draw  [line width=3]  (353.22,97.71) -- (342.91,97.71) -- (342.91,86.31) ;
\draw  [line width=3]  (352.19,75.95) -- (362.5,75.95) -- (362.5,87.35) ;
\draw [color={rgb, 255:red, 255; green, 9; blue, 0 }  ,draw opacity=1 ][line width=1.5]    (345.97,94.38) -- (357.73,81.29) ;
\draw [shift={(360.41,78.32)}, rotate = 491.95] [fill={rgb, 255:red, 255; green, 9; blue, 0 }  ,fill opacity=1 ][line width=0.08]  [draw opacity=0] (9.29,-4.46) -- (0,0) -- (9.29,4.46) -- cycle    ;

\draw (263,67) node [anchor=north west][inner sep=0.75pt]  [font=\scriptsize,color=blue  ,opacity=1 ] [align=left] {$\displaystyle s_{1}$};
\draw (284,34) node [anchor=north west][inner sep=0.75pt]  [font=\scriptsize,color=blue  ,opacity=1 ] [align=left] {$\displaystyle s_{2}$};
\draw (343.34,62.77) node [anchor=north west][inner sep=0.75pt]  [font=\scriptsize,color=blue  ,opacity=1 ] [align=left] {$\displaystyle \iota ( s_{1})$};
\draw (331.75,98.01) node [anchor=north west][inner sep=0.75pt]  [font=\scriptsize,color=blue  ,opacity=1 ] [align=left] {$\displaystyle \iota ( s_{2})$};
\draw (252,16) node [anchor=north west][inner sep=0.75pt]  [font=\footnotesize] [align=left] {$\displaystyle P$};
\draw (358,14) node [anchor=north west][inner sep=0.75pt]  [font=\footnotesize] [align=left] {$\displaystyle \iota ( P)$};

\end{tikzpicture}
\caption{The closed geodesic $\gamma_C$ in $X(\widehat{\Omega})$.}
\label{good:cylinder:fig}
\end{figure}
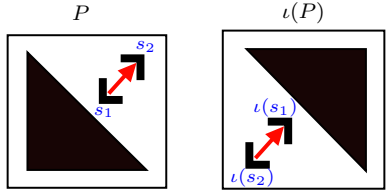

\vspace{5mm}
\noindent
We see that $\iota(\gamma_C) = -\gamma_C$, and so $m(C) = [\gamma_C] \in  H_1^{\iota-}(X,\Sigma,\mathbb{R})$,
giving \ref{thm:main:re}(4).
The idea behind the construction here was to have a closed geodesic passing through a fixed point of $\iota$.
Such a closed geodesic is necessarily $\iota$ anti-invariant.
The problem is that the only fixed points of $\iota$ are vertices. 
For this reason, we needed a regular singularity.

\vspace{5mm}
\noindent
We complete the proof of the theorem.\\
Recall $[\gamma_1] =  \sum_{\theta_i \text{\,rotation}}^{i = 1, \dots, 2k} \, \sum_{j=1}^{d} a_{1}^{j} {\theta_i}_* h\left(e_{j}\right)    -    \sum_{\rho\theta_i \text{\,reflection}}^{i = 1, \dots, 2k} \, \sum_{j=1}^{d} a_{1}^{j} {\rho\theta_i}_* h\left(e_{j}\right)$,
and similarly for $[\gamma_2]$.
\begin{proof}[Proof of Theorem \ref{thm:main:re} items (2) and (3)]
 ~\\
$$\omega(\sum_{\theta_i \text{\,rotation}}^{i = 1, \dots, n} \, \sum_{j=1}^{d} a_{1}^{j} {\theta_i}_* h(e_{j})) = 
\sum_{j=1}^{d}  a_{1}^{j} \sum_{\theta_i \text{\,rotation}}^{i = 1, \dots, n} \omega ({\theta_i}_* h(e_{j}))  = 
\sum_{j=1}^{d}  a_{1}^{j} \sum_{\theta_i \text{\,rotation}}^{i = 1, \dots, n} \theta_i \omega (h(e_{j})). 
$$
But $\sum_{\theta_i \text{\,rotation}}^{i = 1, \dots, n} \theta_i v = 0$ for any $v \in \mathbb{R}^2$, given that $n > 1$.
Similarly, $\sum_{\rho\theta_i \text{\,reflection}}^{i = 1, \dots, n} \rho\theta_i v = \rho(\sum_{\theta_i \text{\,rotation}}^{i = 1, \dots, n} \theta_i v) = 0$.
Thus, $\omega([\gamma_1]) = 0$, and likewise $\omega([\gamma_2]) = 0$.

~\\
Marking $\theta_k = \iota$, we have
 $$\iota_*(\sum_{j=1}^{d}  a_{1}^{j} \sum_{\theta_i \text{\,rotation}}^{i = 1, \dots, n} \theta_{i*} h(e_{j})) = 
 \sum_{j=1}^{d}  a_{1}^{j} \sum_{\theta_i \text{\,rotation}}^{i = 1, \dots, n}\iota_*(\theta_{i*} h(e_{j})) = 
  $$
   $$
 \sum_{j=1}^{d}  a_{1}^{j} \sum_{\theta_i \text{\,rotation}}^{i = 1, \dots, n} \theta_{{i+k(\mathrm{mod}~n)}*} h(e_{j}) = 
 \sum_{j=1}^{d}  a_{1}^{j} \sum_{\theta_i \text{\,rotation}}^{i = 1, \dots, n} \theta_{i*} h(e_{j}).
 $$
 This holds similarly for the second summand, giving $\iota_*[\gamma_1] = [\gamma_1]$, and likewise for $[\gamma_2]$. 
\end{proof}

\end{document}